\documentclass[final]{siamltex}
\usepackage{url}

\usepackage{epsfig}
\usepackage{graphics}
\usepackage{amsmath}
\usepackage{amssymb}
\usepackage{xcolor}
\usepackage[colorlinks=true, hyperfigures=false,
           urlcolor=blue, 
           pdfhighlight =/O]{hyperref}
\usepackage{float}

\def\eps{\varepsilon}

\newtheorem{rem}{Remark}

\DeclareMathOperator{\sgn}{sgn}

\begin{document}
\title{Canards, folded nodes and mixed-mode oscillations in piecewise-linear slow-fast systems}
\author{Mathieu Desroches\footnotemark[1] \footnotemark[6] 
\and Antoni Guillamon\footnotemark[2] 
\and Enrique Ponce\footnotemark[3] 
\and Rafel Prohens\footnotemark[4]
\and Serafim Rodrigues\footnotemark[5] 
\and Antonio E. Teruel\footnotemark[4]} 
\date{}
\maketitle
\renewcommand{\thefootnote}{\fnsymbol{footnote}}
\setcounter{footnote}{1}
\footnotetext[1]{Inria Sophia Antipolis Research Centre, Valbonne, France. \thanks{MD acknowledges the support of the University of the Balearic Islands through a research scholarship.}}
\footnotetext[2]{Department of Applied Mathematics, Polytechnic University of Catalunya, Barcelona, Spain.}
\footnotetext[3]{Department of Applied Mathematics, University of Sevilla, Sevilla, Spain. \thanks{EP is partially supported by by the Spanish Ministerio de Ciencia y Tecnolog{\'i}a in the frame of project number MTM2012-31821, and by the Consejer{\'i}a de Econom{\'i}a-Innovaci{\'o}n-Ciencia-Empleo de la Junta de Andaluc{\'i}a under grant P12-FQM-1658.}}
\footnotetext[4]{Department of Applied Mathematics, University of the Balearic Islands, Palma de Mallorca, Spain. \thanks{RP and AET are supported by a MCYT/FEDER grant number MTM2011-22751 and MTM2014-54275-P}. \newline The authors thank Martin Krupa for useful discussions.}
\footnotetext[5]{School of Computing and Mathematics, Plymouth University, Drake Circus, Plymouth PL4 8AA, UK.}
\footnotetext[6]{{\tt mathieu.desroches@inria.fr}}
\renewcommand{\thefootnote}{\arabic{footnote}}

\begin{abstract}
Canard-induced phenomena have been extensively studied in the last three decades, both from the mathematical and from the application viewpoints. Canards in slow-fast systems with (at least) two slow variables, especially near folded-node singularities, give an essential generating mechanism for Mixed-Mode oscillations (MMOs) in the framework of smooth multiple timescale systems. There is a wealth of literature on such slow-fast dynamical systems and many models displaying canard-induced MMOs, in particular in neuroscience. In parallel, since the late 1990s several papers have shown that the canard phenomenon can be faithfully reproduced with piecewise-linear (PWL) systems in two dimensions although very few results are available in the three-dimensional case. The present paper aims to bridge this gap by analysing canonical PWL systems that display folded singularities, primary and secondary canards, with a similar control of the maximal winding number as in the smooth case. We also show that the singular 
phase portraits are compatible in both frameworks. Finally, we show on an example how to construct a (linear) global return and obtain robust PWL MMOs.
\end{abstract}

\section{Introduction}
\label{intro}

\subsection{Motivations}
One of the most interesting dynamical behaviours exhibited by multiple timescale systems is the existence of trajectories which, after flowing close to an attracting slow manifold, remain close to a repelling slow manifold during a significant amount of time; see~\cite{F79} for the definition and properties of slow manifolds. These trajectories are usually called \textit{canards}~\cite{BCDD81}. When such trajectories stay close to a repelling manifold for as long as this manifold exists, then one speaks of {\it maximal canards}.\newline
In planar slow-fast systems, classical canards arise as limit cycles whose amplitude and period grow in an explosive manner as a parameter is varied~\cite{DR96,KS01b}. A consequence of this property is that the canard phenomenon is short-lived~\cite{diener84} in the planar case. More precisely, canards exist in parameter intervals that are exponentially small in the timescale (small) parameter $\eps$. Furthermore, only one maximal canard can exist, corresponding to the connection between a one-dimensional attracting slow manifold and a one-dimensional repelling slow manifold~\cite{DR96,KS01b}. This fact makes canards delicate to compute numerically~\cite{guckenheimer00} and difficult to relate to experimental observations. Note that canard-type behaviour has been observed experimentally in controlled electronic circuits~\cite{itoh91,tomiyasu90a,tomiyasu90b}. \newline Canard solutions become much more ``robust'' in three-dimensional systems with two slow variables. Indeed, they 
exist for a range of parameters which is not exponentially small in $\varepsilon$. Furthermore, multiple maximal canards can exist along the transverse intersection curves of two-dimensional attracting and repelling slow manifolds. Besides, in this context canards can be responsible for the existence of Mixed-Mode Oscillations (MMOs), that is, trajectories with alternating Small-Amplitude Oscillations (SAOs) and Large-Amplitude Oscillations (LAOs). The SAOs of canard-induced MMOs are controlled by special points called {\it folded nodes}~\cite{guckenheimer05,W05}. This topic has a great relevance in applications, in particular in the context of chemical reactions~\cite{guckenheimer13,milik}, the modelling of cellular electrical and secretory activity~\cite{Ermentrout09,harvey11,krupa08,RW08}, bistable perception \cite{curtu10},
optical oscillations~\cite{marino11}, among other nonlinear phenomena.

There are different ways to obtain MMOs, for instance related to the Shilnikov homoclinic bifurcation~\cite{koper95} and to the break-up of an invariant torus~\cite{larter91}. However, we choose to focus on multiple-timescale  canard-induced MMOs since this topic has received a lot of attention in the recent years. In the context of multiple timescale systems, MMOs are an example of what we will call ``complex oscillations'', which refer to oscillatory solutions with both fast and slow components, different amplitudes and different frequencies. Roughly speaking, there are two main families of complex oscillations: MMOs and bursting oscillations~\cite{desroches13,DK04}. In this paper we exclusively focus on the former ones in PWL systems since the very notion of MMO has not been studied in that framework. Bursting oscillations in PWL systems also constitute an interesting topic that can be related to canards, but we do not address this question in the present work.

PWL systems are known to reproduce nonlinear behaviour (limit cycles, bifurcations, both local and global phenomena) within a mathematical framework that allows for both qualitative analysis and quantitative characterisation of a given system (parameter and period estimation), where one can gain more control on the system without losing any interesting (nonlinear) dynamics. Quantitative information via a PWL approach has been recently used in~\cite{fernandez15,FG15} in the context of hormone secretion, and in~\cite{GPTV15} in the context of the estimation of synaptic conductances. Canard dynamics has been investigated in planar PWL slow-fast systems, from the 1990s~\cite{ARIMA97,KS91} up to very recently~\cite{DFHPT13, FG14, RCG12, SK11}; see Section~\ref{pwl2d}. It is then natural to try to reproduce canard-induced MMO behaviour in three-dimensional PWL slow-fast systems and investigate the equivalent of maximal canards (primary, secondary) and folded nodes (among other {\it folded singularities}). A first step in 
that direction was made in~\cite{PT13}, without a link to folded singularities. 

This work is part of a larger research program aimed at building up a theory for slow-fast dynamics by using PWL systems, and then deriving simplified models that are meaningful for neuroscience applications. In the first part, the goal is both to reproduce complicated smooth slow-fast dynamics in the framework of PWL vector fields and to propose a simplified interpretation of these phenomena. In particular, the concept of slow manifold is defined in a more natural way by using PWL systems; also, the behaviour of canard solutions can be seen more clearly when observing it through a PWL prism. In our opinion, PWL systems offer the optimal framework to keep the essential elements for the rich nonlinear dynamics to emerge while allowing for simplified interpretations. On the application side, the goal is to develop a PWL version of conductance-based models. Indeed, the voltage equation is the result of Kirchoff's law and can be formulated in terms of PWL functions; the equation for the gating variables use sigmoid functions and can also be replaced by PWL functions. This approach has been taken in~\cite{D97,DR03} but only in the planar context. Our objective is to extend it to three-dimensional canard-induced MMO systems (as those in e.g.~\cite{krupa08,RW08}) using the theory developed in the present work. The overarching goal will be to capitalise on the ability of PWL systems to give access to quantitative informations (in terms of parameters, frequency properties, etc..) while keeping the biophysical interpretation of the variables. This research program includes recent and on-going work by the authors and their collaborators~\cite{DFHPT13,FG14,PT13,PTV14}. 

We believe that the results presented in this work constitute an important step in the realisation of the above research program, as it opens the way for a more thorough analysis of 
(at least) three-dimensional PWL systems that display canard dynamics and complex oscillations. Note that there is also recent interest in canard phenomena in piecewise-smooth systems, in the context of Filippov vector fields, both in planar and three-dimensional systems; we refer the reader to, e.g., \cite{KH14,RG14}. The approach there is different since instead of approximating folded manifolds of slow motion by PWL objects, one keeps the smoothness of these manifolds but considers them as switching manifolds between different vector fields. We will not focus on this approach in the present work as we believe that PWL systems offer the simplest framework to study these dynamical phenomena.

\subsection{Review of the literature on PWL slow-fast dynamics}
Our present work contains a substantial amount of new results concerning canard dynamics in a three-dimensional PWL setting. The idea of studying canard phenomena by combining PWL dynamics with multiple timescales dates back to the early 1990s and has had three main phases, which we review below.

\subsubsection{Phase-plane analysis}

The first community that seems to have studied canards in PWL systems was from Japan in the 1990s~\cite{ARIMA97,KS91,itoh91,tomiyasu90b,tomiyasu90a}. Whereas a partial analysis of the planar case was performed, hardly anything was done in the three-dimensional one. Furthermore, the emphasis was very much towards applications to electrical circuits, which also justifies the use of PWL systems. There is indeed a long-standing tradition of using this particular mathematical framework to model electrical and electronic circuits, which dates back to the work of Andronov {\it et al.}~\cite{Andronov66}. In particular, the experimental work done within this Japanese group provided the only examples to date of experimental canards. Overall, the main contributions of this group of papers can be summarised as a phase-plane analysis --- that is, describing important dynamical objects for fixed parameter values --- of PWL canard systems in two-dimensional problems. An important finding reported in~\cite{
ARIMA97} is the 
idea that the 
dynamics near the fold of the {\it critical manifold} (fast nullcline) in van der Pol type systems needs a three-piece PWL approximation in order to reproduce canards. A fourth piece is needed in order to have a global return and to obtain canards with head. Therefore, based on the difference between the approach of Komuro and Saito~\cite{KS91}, and that of Arima {\it et al.}~\cite{ARIMA97}, we will refer to PWL systems 
similar to that studied by Komuro and Saito as {\it two-piece local systems} given that their PWL critical manifold has locally two pieces to approximate the quadratic fold of smooth van der Pol type systems. Likewise, we will refer to PWL systems having true canards as {\it three-piece local systems} given that their PWL critical manifold has locally three pieces to approximate the quadratic fold. Note that the two families of systems just defined may have three or four linearity zones, respectively, the ``last'' zone being used for the global dynamics.

The main argument used to justify this three-piece approximation of a quadratic fold was to obtain the correct transition of eigenvalues near the fold: real, complex, and then real again. In the case where the cubic nullcline is approximated near one of its fold points by one corner, that is, in the case of a two-piece local system, the absence of true canards with still the presence of an explosive growth of cycles upon parameter variation was reported in~\cite{KS91} and recently proven rigorously in~\cite{DFHPT13}. 

\subsubsection{Bifurcation analysis}
Instead of considering only snapshots of the parameter space corresponding to different dynamical regimes, recent studies on planar PWL slow-fast systems have focused on understanding the transitions between these dynamical regimes from a bifurcation viewpoint. Namely, identifying bifurcation scenarios and characterising more systematically the behaviour of emerging canard-like cycles. In this second phase, the main articles are~\cite{DFHPT13,llibre02,RCG12,SK11}. The idea in this group of papers is to revisit the work by the Japanese school and try to clarify the notion of canard cycle and canard explosion in planar PWL slow-fast systems. Applications to neuron models were considered~\cite{RCG12}, also in the context of stochasticity, in particular how noise can induce complex-oscillations in planar PWL canard systems~\cite{SK11}. Both two-piece local systems {--- in} \cite{DFHPT13}, referring to the work done in~\cite{KS91} {---} and three-piece local systems {---} in \cite{RCG12}, referring to the work done in~\cite{ARIMA97} {---} were considered in order to approximate van der Pol type systems in the canard regime. Overall, the main contributions of this school is to have defined in a more theoretical framework what conditions are needed to obtain families of limit cycles that grow explosively, and families where true canard explosions occur, in the context of PWL vector fields with multiple timescales.

\subsubsection{Singular perturbation analysis}
More recent work initiated by Prohens and Teruel in~\cite{PT13} tries to extend singular perturbation theory to the PWL framework. The present work naturally belongs to this category, starting from the smooth situation and extending the results to the PWL case. The main objectives of this emerging approach is to revisit Geometric Singular Perturbation Theory (GSPT) using the PWL framework, in dimension greater than two and in the canard regime. In this way, the planar case can be revisited once more~\cite{FG14}, where the existence of canard cycles in the PWL context is obtained using a similar approach as in~\cite{KS01b}. Furthermore, generalisations to higher dimensions are possible. The present paper deals with three-dimensional PWL systems, with the particular goal of finding the equivalent of folded singularities and associated canards in this context. The work 
done in~\cite{PT13, PTV14} permits to obtain Fenichel type results for PWL systems with one fast variable and $n$ slow variables using a matrix formulation. However, even though this configuration reproduces slow-fast behaviour far away from the corner, it does not capture the passage of the flow through the corner curve and, hence, does not allow for complex oscillations in a PWL framework, a situation to be remedied here. Note that Fenichel manifolds have been previously characterised for singularly perturbed linear systems in the context of control theory (see~\cite{KKO}); however, the authors of this work did not address the question of dynamical passage from an attracting to a repelling Fenichel manifold, which is the main novelty of the results obtained in~\cite{PT13,PTV14}.

\subsection{Main results of the paper}
In the planar case, all systems that we will consider can be written in the Li{\'e}nard form
\begin{eqnarray}
\label{lie}
  \begin{array}{l}
    \eps\dot{x} = y - f(x),\\
    \;\;\dot{y} = a - x, 
  \end{array}
\end{eqnarray}
where $f$ is a PWL continuous function that makes system~\eqref{lie} a two-piece local system or a three-piece local system. 
In the three-dimensional case, all systems that we will consider can be written in the form
\begin{eqnarray}
\label{3d}
  \begin{array}{l}
    \eps\dot{x} = - y + f(x),\\
    \;\;\dot{y} =  p_1x + p_2z,\\ 
    \;\;\dot{z} =  p_3,
  \end{array}
\end{eqnarray}
where $p_i$ are parameters and $f$ is as before. The choice of signs in the fast equation of the three-dimensional system is made in order for the resulting PWL system to be written in the exact same form as the minimal smooth systems for folded singularities~\cite{P08}. From the last equation of~\eqref{3d}, it is easy to see that such a system cannot have periodic orbits for $p_3\neq 0$. Our results focus mostly on the local dynamics near folded singularities, but we also show in Section~\ref{pwlmmo} how to construct a global return by adding extra terms to the $\dot{z}$-equation. We next highlight in more detail the main findings of this work.

\subsubsection{Local dynamics}
We construct three-dimensional minimal PWL slow-fast systems (with two slow variables) displaying the same (local) dynamics as the smooth minimal systems for folded singularities. We analyse the singular dynamics in all the cases that lead to persistent canards (that is, we do not consider the folded focus case) and show that the resulting singular phase portraits are compatible with the smooth case. We provide the first interpretation (in terms of GSPT) of the central piece of the critical manifold introduced by Arima {\it et al.} in~\cite{ARIMA97} as a ``blow-up'' of the corner curve (equivalent of the fold curve in the PWL setting). What is more, our analysis shows that the optimal size (in a sense to be made more precise later) of this central part is a function of the small parameter $\eps$, which agrees with the size of the blow-up performed in the analysis of the smooth case. Under 
suitable assumptions, we describe the dynamics for $\eps>0$ in the case of folded node --- which is the most 
important one, from the MMO 
viewpoint --- and folded saddle. We also report findings in the PWL 
context that can help to revisit the smooth case. Indeed, we report the possibility of SAOs near a folded saddle, both in the PWL and in the smooth framework, which seems to have been unnoticed in the smooth case. Furthermore, we show that the so-called {\it weak canard} for $\eps>0$ in the folded node case is not a maximal canard in the PWL context; to the best of our knowledge, this question does not seem to be entirely resolved in the smooth case. Then our result with the PWL framework suggests to revisit the role and properties of the equivalent trajectory in the smooth case.

\subsubsection{Global dynamics}
We propose an example of global return for a PWL slow-fast system with a folded node, so that we can construct robust canard-induced MMOs. This result is of interest in neuroscience as the frequency properties of such PWL MMO neuron models can be quantitatively controlled. In order to exhaust the comparison with the smooth case, in particular with neuronal MMO models, a  more complete analysis of the possible MMO patterns in the PWL context (such as Farey tree structure, devil staircase, or chaotic behaviour) is among the immediate future objectives of our research program.

\subsubsection{Openings towards new directions of research}
Through the present work, we wish to promote two avenues of research. First, we aim to develop an extension of GSPT within the framework of PWL slow-fast systems; this approach has been initiated in~\cite{FG14,PT13,PTV14}. Within this extended theory, the gain is to obtain similar results with a simplified presentation and clarified definitions (in particular regarding the uniqueness of the essential dynamical objects). Besides, we want to particularly focus on complex canard-induced oscillations. Indeed, the present work is mostly concerned with MMOs, but our immediate goal is to revisit canard-induced bursting oscillations in a PWL setting. Second, we provide an important step towards building biophysical PWL models of neurons displaying complex oscillations, canard-induced MMOs and bursting oscillations. These can then be analysed using singular perturbation theory, just like a number of conductance-based smooth models have been analysed in the recent years~\cite{DGKKOW,Ermentrout09,harvey11,krupa08,
rubin02,W05}. The objective is to eventually obtain PWL versions of Hodgkin-Huxley type models that are both biologically meaningful and more amenable to analysis.
As polynomial systems such as the FitzHugh-Nagumo model capture the mechanisms of more complex neuronal models, our results on model (1.2) can also pave the way to understand the generation of MMOs in Hodgkin-Huxley type models~\cite{RW08}.

\subsection{Outline of the paper}
In Section~\ref{pwl2d}, we review planar PWL slow-fast systems, splitting the study into two parts. In the first part, we review two-piece local systems~\cite{KS91}, where an explosive growth of limit cycles can occur even though true canards are not possible (hence termed {\it quasi-canards} in~\cite{DFHPT13}). In the second part, we consider the case of three-piece local systems (\cite{ARIMA97}) where true canard explosions can occur. We present the main previous work on these cases as well as more recent results for PWL slow-fast systems both with three-piece and four-piece critical manifolds, only the latter ones giving rise to canard solutions related to those in smooth systems. In Section~\ref{3dslowpas}, we introduce our strategy to construct canard-type dynamics in three-dimensional PWL slow-fast systems, by putting a slow drift onto the parameter that induces the canard transition obtained in planar two-piece and three-piece local systems. Direct simulation of such systems naturally reveal transitory MMOs that can be interpreted as a quasi-canard explosion 
with a slow drift and as a canard explosion with a slow drift, respectively. The 
latter case is the most interesting one for our purpose because we want to investigate the equivalent of folded singularities and the associated dynamics (in terms of maximal canards) in the PWL context. 
Therefore, in Section~\ref{pwlfsing}, we analyse the local dynamics created by a canard explosion with a slow drift and we find the PWL equivalent of all the key objects that are present in the smooth case. Namely, folded singularities, singular weak canard and singular strong canard for $\eps=0$, and persistent canards for $0<\eps\ll 1$, that is, weak canard $\gamma_w$ and strong canard $\gamma_s$ in the folded-node case, faux canard $\gamma_f$ and strong canard $\gamma_s$ in the folded-saddle case; see Figure~\ref{fcr}. We investigate thoroughly the cases of folded saddle and folded node, and we also provide the singular phase portraits for the two folded saddle-node cases. In Section~\ref{pwlmmo}, we show how to construct a global return near a PWL folded node, so that we can create robust canard-induced MMOs similar to what can be done in the smooth case~\cite{brons06,DGKKOW}. Finally, in the conclusion section, we propose two main perspectives that our work unveils. 
\section{{Planar PWL slow-fast systems}}
\label{pwl2d}

\subsection{Quasi-canards}
\label{quasican}
A large class of neuron models (for action potential generation) are based on the approximation that the membrane of a neuron behaves like an electrical circuit. Therefore, the voltage equation in such models is obtained by applying Kirchoff's law. The other dynamical equations account for the exchange of ions across the neuron membrane as ion channels open and close, which organises the generation of an action potential. After the model by the Nobel prize winners A. Hodgkin and A. Huxley~\cite{HH52} (HH), the first reduction from the original four-dimensional model to a two-dimensional model was made independently by R. FitzHugh~\cite{fitzhugh61} and J. Nagumo~\cite{nagumo62} in the early 1960s. In addition to dimension reduction, the vector field of the HH model was approximated by a polynomial system through the crucial observation that the voltage nullcline is roughly cubic shaped. Hence, the FitzHugh-Nagumo (FHN) model appears as a modified van der Pol system, thus prompt to be investigated from the 
slow-fast perspective. At the beginning of the 1970s, 
McKean~\cite{mckean70} further simplified the FHN model by 
approximating the cubic voltage nullcline by a PWL function. Similar to the HH- and FHN-models, the McKean model comes originally from a reaction-diffusion PDE in order to study the speed of propagation of spiking solutions; then many results on traveling pulses were obtained in this context~\cite{tonnelier03}. McKean also studied the ODE version of his model in~\cite{mckean70}, in the diffusion-free regime. Abbott~\cite{abbott90} extended this work on PWL relaxation oscillations by estimating the period, and coupling two such systems.

After a brief mention of the ``loss'' of canards in PWL systems with two corners in~\cite{itoh91}, the first study of a PWL van der Pol system from the perspective of canards was made by Komuro and Saito in 1991~\cite{KS91}. They performed a phase-plane analysis on what is essentially the ODE McKean model and they showed numerically that the transition from the equilibrium to the relaxation cycle was very rapid in terms of parameter variation and involved small canard-like cycles; see Figure~\ref{ksfig}(a). They also reported the absence of cycles resembling canards with head as well as the fact that the slope of the middle segment constrains how small $\eps$ can be. These findings were recently clarified in~\cite{DFHPT13}. In particular, the authors explained the reason for the absence of canards, namely, the fact that there is no repelling slow manifold in the system. However, the name {\it quasi-canard} was given to the cycles observed by Komuro and Saito, precisely because of their shape and their 
explosive behaviour upon parameter variation. Furthermore, the authors of~\cite{DFHPT13} 
showed that the small parameter $\eps$ is indeed constrained by the value that the slope of the central segment takes, but one can still make $\eps$ as small as possible while suitably adjusting this slope. Also, they analysed the limit of the quasi-canard regime in the parameter plane. In particular, as the equilibrium in the central zone reaches the corner (exactly on the switching line), the system admits a 
continuum of true homoclinic canard connections, which motivated the term {\it super-explosion} in order to 
explain the discontinuous transition from the stationary regime to the relaxation regime; see~\cite{DFHPT13} for detail.
\begin{figure}[!t]
\centering
\includegraphics[width=11cm]{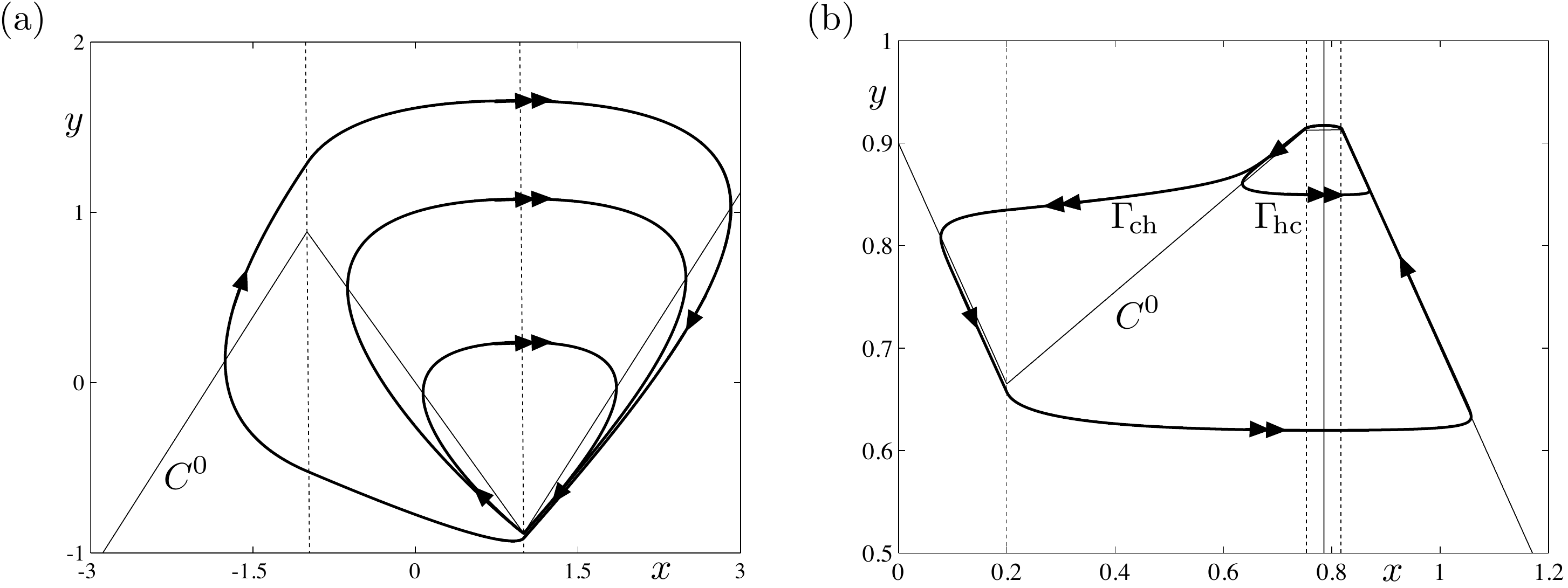}
\caption{\label{ksfig}(a) Quasi-canard solutions of the two-piece local system~\eqref{lie} with a three-piece PWL critical manifold $y=f(x)$ where $f(x)=x+\frac 1 2(1+k)(|x-1|-|x+1|)$ and parameters $\varepsilon=0.2$, $k=0.885$ and $0.5 \leq a <1$. (b)Two canard solutions (a headless canard $\Gamma_{\!\mathrm{hc}}$ and a canard with head $\Gamma_{\!\mathrm{ch}}$) of the three-piece local system from~\cite{ARIMA97} with a four-piece PWL critical manifold. The vector field has been chosen exactly as in~\cite{ARIMA97}, that is, taking the form of~\eqref{pwlfunc4} but with different slopes in different zones. We have also taken the critical manifold to be flat in the central zone. The values of the parameter controlling the slow nullcline (parameter $a$) for these two canard cycles are: $a=0.783913238$ for the headless canard $\Gamma_{\mathrm{hc}}$ and $a=0.783913236$ for the canard with head $\Gamma_{\mathrm{ch}}$. In both panels, $C^0$ denotes the PWL critical manifold.}
\end{figure}

\subsection{Canards}
\label{2dcan}
The main idea to obtain true canard cycles in a planar PWL systems was introduced by Arima {\it et al.} in 1997~\cite{ARIMA97}. It consists in approximating the critical manifold near a fold by a three-piece PWL function. The justification given in~\cite{ARIMA97} is to reproduce the transition between real and complex eigenvalues that occurs near the fold point in the smooth system. The authors of~\cite{ARIMA97} also notice that, due to the presence of this extra linearity zone, a separatrix appears (as an invariant manifold of a virtual equilibrium) which allows for both canards without head and canards with head to exist; this separatrix provides a canonical choice of repelling slow manifold. We complement their interpretation by the following argument. Arima {\it et al.} based their work on the previous work by Komuro and Saito, who considered a cubic shaped PWL critical manifold (two-piece local system), with a focus equilibrium in the  central zone. Therefore, there is no invariant manifold of 
equilibrium that can play the role of repelling slow manifold and, hence, allow for 
small and large canards to exist. Only small quasi-canards can occur in this situation. However, instead of the focus, one could consider a node equilibrium in the central zone, which would create a repelling slow manifold (one of the stable manifolds of the node). But the problem here would be that no small limit cycle can exist, the cycles are born large (three-zonal). So, in order to combine both the possibility for small cycles and for a repelling slow manifold, one needs the central zone which introduces complex eigenvalues, and the virtual node whose stable manifold provides a repelling slow manifold to the system. Therefore, Arima {\it et al.} found, 
without providing a complete mathematical understanding, the minimal PWL setting to construct a canard-explosive system.

Figure~\ref{ksfig}(b) shows a canard cycle without head and a canard cycle with head in the Arima {\it et al.} model. Note that the authors of~\cite{ARIMA97} claim that the critical manifold in the central zone has to have a non-zero but small-enough slope, which in fact is not true. The only important point is to have complex eigenvalues, therefore one can choose the critical manifold in the central zone to be flat. This is the case in the simulations shown on Figure~\ref{ksfig}(b). When analysing a three-dimensional version of this model, in Section~\ref{3dslowpas} and~\ref{pwlfsing}, we will also keep that slope equal to zero in the central zone.

The work by Arima {\it et al.} was mostly concerned with phase-plane analysis. In the recent years, various works have been done to revisit and complement their results by incorporating some bifurcation analysis. In particular, Pokrovskii {\it et al.}~\cite{PV05} recover (without mentioning it) the main results from~\cite{ARIMA97}, that is, the fact that three pieces are needed to approximate the critical manifold near a fold point, in order to create canard cycles, and that a fourth one allows for canards with head to exist. Also, Rotstein {\it et al.}~\cite{RCG12} obtain bifurcation diagrams corresponding to both supercritical and subcritical canard explosion in this context. Furthermore, they  make the link between canards and excitable dynamics with threshold crossing, in the context of neuron models.

The latest work on this topic is by Fern{\'a}ndez-Garc{\'i}a {\it et al.}~\cite{FG14}. In this work, the main idea of Arima {\it et al.} is revisited once more --- without considering the  fourth zone, only locally near the three-piece PWL approximation of the quadratic fold  ---, from yet a different viewpoint. The main focus is to extend singular perturbation results from the smooth case~\cite{KS01b} in order to prove the existence of maximal canards, as connections between attracting and repelling slow manifolds, and related families of canard cycles (break-up of such connections) in this context. This provides the first singular perturbation analysis of planar PWL canard systems. The present work takes this approach to three-dimensional PWL slow-fast systems, although the presentation is less technical.

Now that results regarding canard cycles in planar PWL slow-fast systems have been reviewed, we can develop the main objective of this paper, that is, to construct and analyse three-dimensional PWL slow-fast systems displaying canard dynamics, especially near folded singularities.
\section{Three-dimensional PWL slow-fast systems: transient MMOs}
\label{3dslowpas}
Once {the key elements allowing the existence of canard cycles in planar PWL systems have been established}, it is natural to consider three-dimensional models. The simplest way to do so is to put a slow drift on the parameter that displays the canard (or quasi-canard) transition in the planar system.  This allows to construct transient canard trajectories in three-dimensional systems, building up on the knowledge from the planar case~\cite{W05}. The problem of making these canards robust (recurrent) is a separate one and requires to construct a global return mechanism, which we will investigate through an example in Section~\ref{pwlmmo}. Below we simply show how to observe numerically a quasi-canard and  a canard explosion, respectively, with a slow drift on the associated parameter. One can find a handful of papers dealing with canards in three-dimensional PWL slow-fast systems~\cite{PT13,PTV14}. However, none of these papers deal with planar canard dynamics with a slow drift, except in the context of 
stochastic systems~\cite{SK11}.

We will now follow a presentation similar to that adopted for the planar case, that is, distinguishing between two-piece local systems and three-piece local systems to approximate a quadratic fold of a smooth slow-fast system. We will simply add a trivial slow dynamics on the parameter displaying the explosion in the planar system. To do that we consider the slow drift 
\begin{eqnarray}
\label{lie-a} \dot{a} &=& c,
\end{eqnarray}
with $c$ a real parameter.

\begin{figure}[!t]
\centering
\includegraphics[width=11cm]{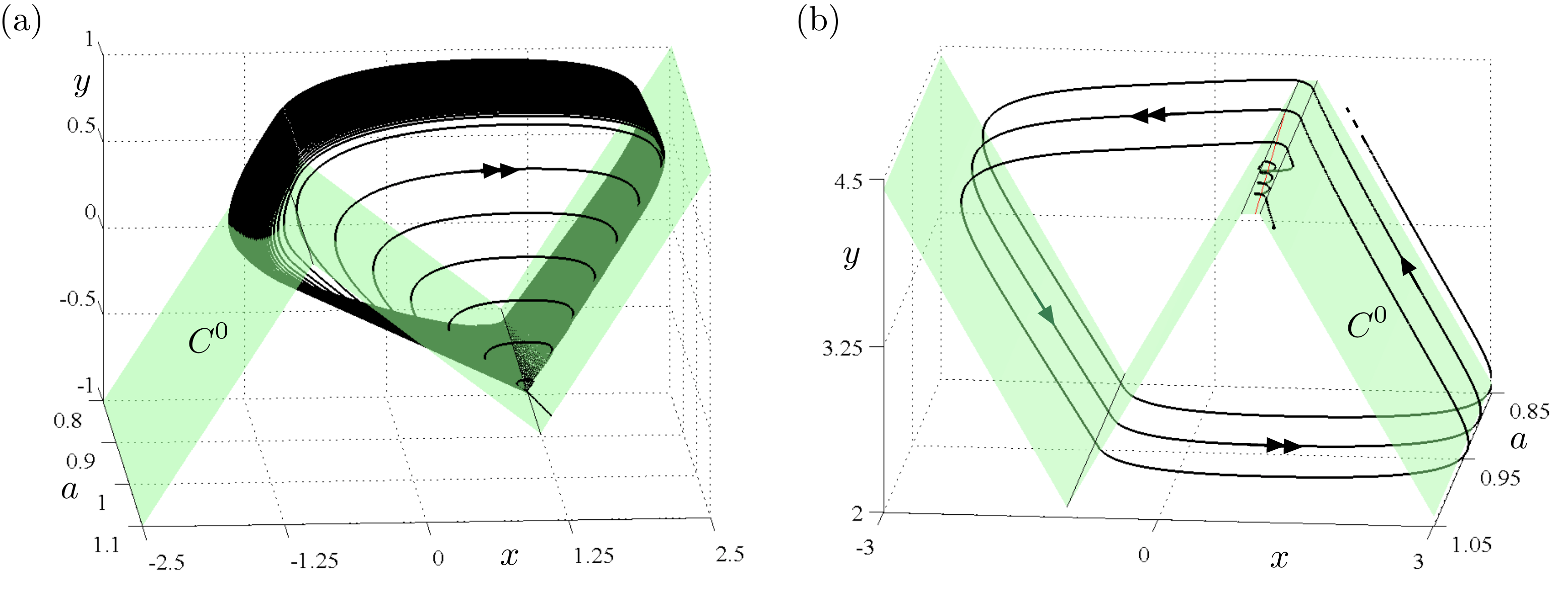}
\caption{\label{ks3dfig}(a) Transient MMO in a three-dimensional version of the two-piece local system~\eqref{lie}; the generating mechanism here is a quasi-canard explosion with a slow drift \eqref{lie-a}. The parameter values for this transient MMO trajectory are: $\eps=0.1$, $k=0.5$, $c=-0.001$. (b) Transient MMO trajectory $\Gamma$ solution to the system from~\cite{ARIMA97} with a drift~\eqref{lie-a} on the slow nullcline. In both panels, $C^0$ denotes the PWL critical manifold.}
\end{figure}
\subsection{Dynamic quasi-canard explosion}
\label{slowpasquasican}
We first consider system~\eqref{lie} with $f(x)=x+\frac 1 2 (1+k) (|x-1|-|x+1|)$ to which we append the slow drift \eqref{lie-a}.
Without any surprise, we can observe a dynamic quasi-canard explosion by direct simulation of the resulting system~\eqref{lie}--\eqref{lie-a}; see Figure~\ref{ks3dfig}(a) with a three-dimensional phase-space view of such a trajectory. This system has the same similarities and the same missing points when compared to three-dimensional smooth canard systems, as the planar system studied by Komuro and Saito when compared to the van der Pol oscillator. That is, an explosive behaviour in the growth of small oscillations but no repelling slow manifold. Overall, one can create transient MMO dynamics but not of true canard type. Therefore the control of SAOs will be different than in the smooth case near a folded node, and the link to existing smooth models is more tenuous. Consequently, we will not investigate this type of MMOs further as we wish to focus on canard-induced MMOs, that is, we want to understand the 
mechanism to create folded 
singularities in three-dimensional PWL slow-fast systems. This is why, as a first step, we will now consider a trivial slow dynamics on the parameter displaying the canard explosion in the system studied by Arima {\it et al.}.

\subsection{ Dynamic canard explosion}
\label{slowpascan}
We consider system~\eqref{lie} with the additional equation~\eqref{lie-a} and with a PWL function $f$ now defined in four zones, that is, $f=F_{\delta}$ where 
\begin{equation}\label{pwlfunc4}
   F_{\delta}(x) = \left\{
     \begin{array}{ll}
	-x+(\beta +1)\delta 		&  \mathrm{if}\quad  x \geq \delta, \\  
        \,\beta x    			&  \mathrm{if}\quad |x| \leq \delta,\\
        \,\;\;x-(\beta-1)\delta 		&  \mathrm{if}\quad x_0<x<-\delta,\\
       -x+2 x_0-(\beta-1)\delta 		&  \mathrm{if}\quad x  \leq x_0.
     \end{array}
   \right.
\end{equation}
This creates canard-induced MMOs in transient dynamics exactly as in the smooth case; see Figure~\ref{ks3dfig}(b) with a three-dimensional phase-space view of such a trajectory.  We (naturally) observe a dynamic canard explosion, and hence, folded node type dynamics. Therefore, we are now confident that we have the correct framework to find where the equivalent of the folded node is, and more generally, how to understand folded singularities in a PWL context. In the next section, we will focus on a three-piece local system in order to unravel the local dynamics near the PWL equivalent of folded singularities, that is, not only folded node but also folded saddle and folded saddle-node.

\section{Three-dimensional PWL slow-fast systems: local dynamics near folded singularities}
\label{pwlfsing}
Let us now analyse the minimal system to reveal PWL folded singularities, whose equations are given by~\eqref{3d} with $f=f_{\delta}$, where
\begin{equation}\label{pwlfunc}
   f_{\delta}(x) = \left\{
     \begin{array}{ll}
        0          & \mathrm{if}\quad |x| \leq \delta,\\
        |x|-\delta & \mathrm{if}\quad |x| \geq \delta.\\
     \end{array}
   \right.
\end{equation}
This system has a three-piece critical manifold, with a central part that is flat and of size $2\delta>0$. At first, the value of $\delta$ is free but it will be constrained after some initial simulations of the system.
We remark that all the results persist with a more general PWL function; in particular, making the slope of the critical manifold in the central part non-zero does not fundamentally alter the overall dynamics. Importantly, it does not remove the possibility for SAOs, the amplitude of which will not be constant because the invariant cylinders to be described in the next section do not exist anymore. We have chosen to consider a flat piece for the critical manifold in the central zone in order to have a minimal model that features all the dynamics that we are interested in.

\subsection{General setting}\label{genfsing}
We start with a few general remarks about the global behaviour of system~\eqref{3d} in the PWL setting, that is, with the function $f$ given by~\eqref{pwlfunc}; see Figure~\ref{fcr} for a pictorial illustration of these remarks.
In order to understand the dynamics of system~\eqref{3d} in the outer zones, that is, for $|x|>\delta$, we make use of the results from~\cite{PT13}. As in this previous work, we obtain expressions for the attracting and the repelling slow manifolds as invariant half-planes spanned by the eigenvectors associated with the slow eigenvalues (one of these eigenvalues being zero); their invariance follows from checking a simple algebraic condition. In particular, the attracting slow manifold is given by
\begin{eqnarray}
\label{attrslowm}S^{A}_{\eps} &=& \left\{x<\;-\delta,\;-\lambda^2_{A}x+\lambda_{A}y + \eps p_2z\;=\;-\delta\lambda_{A} - \frac{ p_2p_3}{\lambda_{A}}\eps^2  \right\},
\end{eqnarray}
with $\lambda_A=-\frac{1+\sqrt{1-4\eps p_1}}{2}$, and the repelling slow manifold is given by
\begin{eqnarray}
\label{repslowm}S^{R}_{\eps} &=& \left\{x>\;\delta,\;-\lambda^2_{R}x+\lambda_{R}y + \eps p_2z\;=\;-\delta\lambda_{R} - \frac{ p_2p_3}{\lambda_{R}}\eps^2 \right\},
\end{eqnarray}
with $\lambda_R=\frac{1+\sqrt{1-4\eps p_1}}{2}$; see~\cite{PT13} for details. Note that due to the symmetry of~\eqref{pwlfunc} we have $\lambda_R=-\lambda_A$.

We now describe the dynamics in the central zone, that is, on and in between the two switching planes $\{x=\pm\delta\}$. In this zone, the fast equation reduces to $\eps\dot{x}=-y$, which implies that the direction of the flow on the switching planes depends upon the sign of $y$. Moreover, the lines $\{x= -\delta, y=0\}$ and $\{x= \delta, y=0\}$ correspond to the locus of tangency points.
On the other hand, it is easy to check that the  function of the state variables 
\begin{equation}\label{axis}
H(x,y,z) = \eps p_1(p_1x+p_2z)^2+(p_1y-\eps p_2p_3)^2
\end{equation}
is a first integral for the system in the central zone.

If $p_1$ is negative, the level sets of $H$ are constituted by two disjoint surfaces, invariant by the central flow, whose sections with $x$ equal to constant are two branches of hyperbola with a common center on the axis given by
\begin{eqnarray}
\label{axis-xy} x = -\frac{p_2}{p_1}z,\quad y = \frac{\eps p_2p_3}{p_1}.
\end{eqnarray}
Moreover, the matrix defining the vector field in the central zone does not have complex eigenvalues. Therefore, we conclude that, when $p_1<0$, no rotation can happen in this zone.\newline In contrast, when $p_1>0$, the level sets of $H$ are cylinders, invariant by the central flow, with a common axis given by \eqref{axis-xy}. Furthermore, the non-zero eigenvalues of the matrix defining the flow in the central zone are $\pm i\, \sqrt{\varepsilon p_1}$, therefore trajectories do rotate in this zone. Indeed, the line segment~\eqref{axis-xy} organises the dynamics of the full system by acting as an axis of rotation for trajectories that display Small-Amplitude Oscillations (SAOs) in the central zone, which corresponds to the so-called {\it weak canard} in the smooth case~\cite{SW01,W05}. In the remainder of the paper, we will always assume $p_1>0$ when considering SAO dynamics. The dynamics for $p_1>0$ in the vicinity of the central zone, together with important invariant objects, is 
sketched in Figure~\ref{fcr}. 
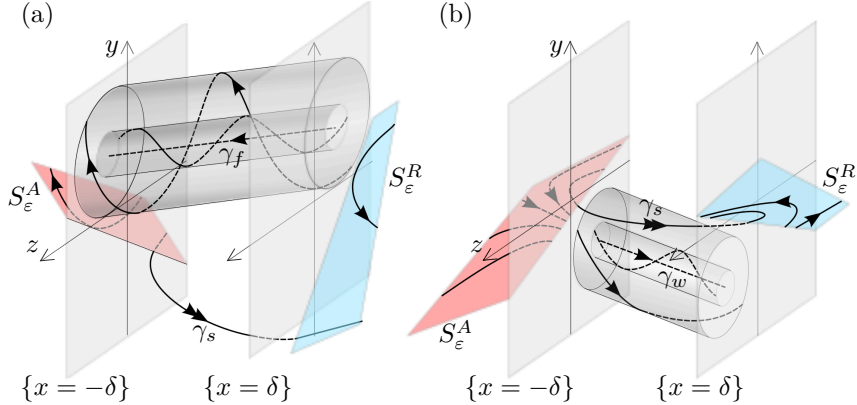
\begin{figure}[!t]
\begin{center}
\includegraphics[width=11cm]{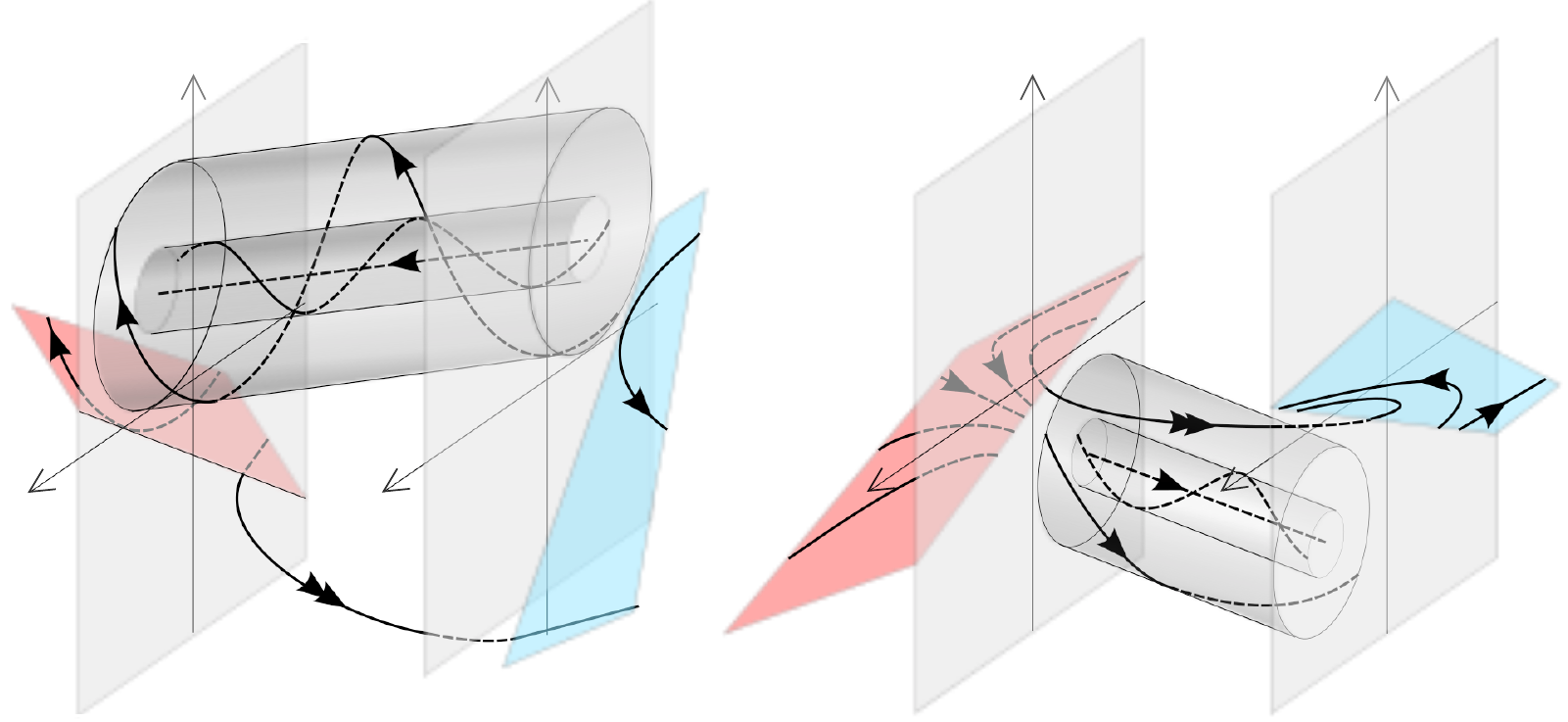}
\begin{picture}(0,0)
 \put(-318,136){(a)}
 \put(-316,49){$z$}
 \put(-286,125){$y$}
 \put(-323,67){$S^A_{\eps}$}
 \put(-179,74){$S^R_{\eps}$}
 \put(-253,17){$\gamma_s$}
 \put(-318,-5){$\{x=-\delta\}$}
 \put(-249,-5){$\{x=\delta\}$}
 \put(-243,84){$\gamma_f$}
 \put(-160,136){(b)}
 \put(-118,125){$y$}
 \put(-148,49){$z$}
 \put(-160,15){$S^A_{\eps}$}
 \put(-15,74){$S^R_{\eps}$}
 \put(-77,37.5){$\gamma_w$}
 \put(-84,64.5){$\gamma_s$}
 \put(-150,-5){$\{x=-\delta\}$}
 \put(-78,-5){$\{x=\delta\}$}
\end{picture}
\caption{\label{fcr}Sketch of the flow given by system \eqref{3d}-\eqref{pwlfunc} with $p_1>0$: attracting slow manifold ($S^A_{\eps}$), repelling slow manifold ($S^R_{\eps}$), weak canard ($\gamma_w$), faux canard ($\gamma_f$) and strong canard ($\gamma_s$). Panel (a) describes the case of a folded saddle, and panel (b) that of a folded node. See Section~\ref{fsad} and~\ref{fnod} for details.}
\end{center}
\end{figure}
In particular, the 
attracting slow manifold is shown as a red plane in the left 
zone, the repelling slow manifold as a blue plane in the right zone, and two level sets of the function $H$ given by equation~\eqref{axis} are drawn in the central part together with a few phase trajectories (black lines). Two cases are sketched, corresponding to the folded saddle scenario (panel (a)) and to the folded node scenario (panel (b)), and special solutions are highlighted. Some of these solutions are canards, such as $\gamma_w$ and $\gamma_s$ on panel (b). Since the one on panel (a) (labelled $\gamma_f$) has the opposite direction, that is, it goes from the repelling slow manifold to the attracting one, it is called a {\it faux canard}. See Section~\ref{foldsing} below for details about these scenarios in the PWL context.

The number of SAOs increases as trajectories approach the axis of rotation. Given that the angular velocity of trajectories is constant, equal to $\sqrt{\eps p_1}$ (since the system is the central zone is linear), the number of complete turns is bounded by the time needed for an orbit starting on the axis of rotation to flow between the switching planes $\{x=\pm\delta\}$. This time can easily be computed since the $z$ equation of system~\eqref{3d}-\eqref{pwlfunc} is just a constant drift. From~\eqref{3d} and~\eqref{axis-xy}, it is easy to see that on the rotation axis $$ \dot{x} = -\frac{p_2}{p_1}\dot{z}= -\frac{p_2 p_3}{p_1}\varepsilon.$$ Thus, the time at which an orbit starting at the rotation axis reaches one switching plane, having started on the other, is given by
\begin{eqnarray}\label{def.max_flight_time}
t^{\ast} &=& \frac{2p_1\delta}{\eps |p_2 p_3|}.
\end{eqnarray}
Therefore, the associated maximal winding (or rotation) number $\mu$ is obtained as 
\begin{equation*}
\mu = \frac{\delta}{\pi\sqrt{\eps}} \frac{p_1\sqrt{p_1}}{|p_2p_3|}.
\end{equation*}
Note that the quantity $\mu$ is reminiscent of the eigenvalue ratio at a folded singularity in the smooth setting~\cite{SW01,W05}. Even if there is no direct correspondence between these two quantities, they play the same role in each of the two frameworks, namely to control the maximal number of SAOs that can occur in the system. Moreover, in the smooth case, this maximal winding number is independent of $\eps$. Thus, in order to reproduce quantitatively the behaviour observed in the smooth context, we choose the size of the central part $\delta$ to be a function of $\sqrt{\eps}$. Note that this choice is clearly suggested by the expression obtained for $\mu$. Therefore, from now on, we fix 
\begin{equation}\label{delta}
\delta=\pi\sqrt{\eps}
\end{equation}
and consequently the maximal winding number $\mu$ is constant and satisfies
\begin{equation}\label{maxrotnum}
\mu=\dfrac {p_1 \sqrt{p_1}}{|{p_2 p_3}|}.
\end{equation}
The choice that we make for $\delta$ is necessary in order to obtain, in a PWL framework, a complete match with the behaviour of smooth slow-fast systems near folded singularities, that is, not only qualitative but also quantitative in the sense that the maximal winding number $\mu$ is finite, independent of $\varepsilon$ and related, via a scaling factor, to the corresponding quantity from the smooth case. An important consequence of the $\eps$-dependence of $\delta$ is that the central part of the PWL critical manifold collapses to a single corner line in the singular limit $\eps=0$, that is, the three-piece local system for $\eps>0$ converges, in the singular limit, to a two-piece local system. 

A direct consequence of the above calculations is that one can see the central zone, needed to obtain canard dynamics, as a blow-up of the corner line that exists in the singular limit. It is also remarkable that the size of this blow-up, $O(\sqrt{\eps})$, matches that of the smooth case. Indeed, when blow-up is performed near non-hyperbolic points in smooth slow-fast systems, it can be proven that the region of hyperbolicity, where canards are shown to exist, is extended in the blown-up locus by a size of $O(\sqrt{\eps})$; see~\cite{KS01b}. We believe that this interpretation of the central zone as a blow-up of the corner line (that appears in the singular limit) is novel and relates to the analysis of smooth slow-fast systems near non-hyperbolic points. Besides, this kind of blow-up seems simpler than its counterpart from the smooth framework in the sense that it does not increase the dimension of the system but 
just adds another linearity zone.
A more geometrical yet complementary view on this central zone is the idea that the linearisation of a parabola by a corner provides a good approximation away from the fold, however near the fold three linear pieces are needed in order to avoid losing too much information.

\subsection{Singular flow}
The fact that the central part of the PWL critical manifold shrinks to a single corner line in the limit $\eps=0$ implies that the singular dynamics of system~\eqref{3d}-\eqref{pwlfunc} in the PWL setting is restricted to two half-planes, one for the attracting part of the dynamics and one for the repelling part. Therefore, the information about the connection from one part to the other is lost. To overcome this difficulty, we propose to keep track of the limit of the central dynamics, as $\eps$ tends to $0$, while maintaining the presence of the central zone; we will say that the central zone stays ``open'' in the singular limit. Consequently, the reduced system that we will consider in order to derive the slow flow is the limit when $\eps=0$ of system~\eqref{3d} with $f=f_{\tilde{\delta}}$, for a certain $\tilde{\delta}>0$ independent of $\eps$. Then, we can apply the usual procedure to find the slow flow. Namely, we set $\eps=0$ in~\eqref{3d}, which yields the following differential-algebraic system
\begin{eqnarray}
\label{red-x}-y+f_{\tilde{\delta}}(x) &=& 0,\\
\label{red-y}\dot{y} &=& p_1x+p_2z,\\
\label{red-z}\dot{z} &=& p_3.
\end{eqnarray}
As it is customary with slow-fast systems in the slow (singular) limit, we differentiate the algebraic equation~\eqref{red-x} with respect to time, and project the overall reduced system onto the  $(z,x)$-plane. Restricting our attention to points not on the separation planes, this brings
the following planar reduced system
\begin{eqnarray}
\label{red2-x}-f'_{\tilde{\delta}}(x)\dot{x} &=& -(p_1x+p_2z),\\
\label{red2-z}\dot{z} &=& p_3.
\end{eqnarray}
In the smooth case~\cite{DGKKOW}, the reduced system is singular along the fold curve of the critical manifold, defined by $\{f'(x)=0\}$ in the formulation of system~\eqref{3d}. In order to overcome this problem, one typically rescales the time by a factor $-f'(x)$, thus obtaining a regular system with the possibility for equilibria on the fold curve provided that $p_1x+p_2z$ also vanishes. In the original system, this corresponds to the cancellation of a simple zero in the numerator and denominator of the right-hand side of the $x$ equation. The so-called {\it Desingularised Reduced System (DRS)} then has a true equilibrium on the fold curve whereas the original reduced system does not. This is because the time rescaling reverses the orientation of trajectories when $f'(x)>0$, that is, along the repelling sheet of the critical manifold. Such an equilibrium is therefore called a {\it folded singularity} (or {\it pseudo-equilibrium}~\cite{benoit90}) and, according to its topological type as an 
equilibrium of the DRS, one speaks of {\it folded node}, {\it folded saddle}, {\it folded focus}.

It is important to note that the singular nature of flow of the reduced system along the fold curve can be resolved at one point, and that point is the folded singularity. Moreover, such a point appears (because of the time rescaling performed for the desingularisation) as a way to pass from the attracting to the repelling sheets of the critical manifold, which defines {\it singular canards}. There can be infinitely many such canards in the limit $\eps=0$ (this is the case for folded node). However, they are organised by special singular canards obtained from the two eigendirections of the folded singularity (provided it has real eigenvalues) and called {\it singular weak canard} and {\it singular strong canard} according to the modulus of the associated eigenvalue. 

\subsubsection{Singular weak canard}
Back to the PWL context, since $f'_{\tilde{\delta}}$ is a piecewise-constant function, equation~\eqref{red2-x} does not require any desingularisation procedure as in the smooth case. On the contrary, it gives a regular system in each zone. Namely, in the outer zones the reduced system has the following form
\begin{eqnarray}
\label{drspwlxy}
  \begin{array}{l}
    \dot{x} = \sgn(x)(p_1x+p_2z),\\
    \dot{z} = p_3,
  \end{array}
\end{eqnarray}
with $x<-\tilde{\delta}$ for the attracting (lower) zone and $x>\tilde{\delta}$ for the repelling (upper) one; see, for instance, Figure~\ref{fsadfig1}. In the central part, the resulting system is more degenerate and its equations are given by
\begin{eqnarray*}
	0 &=& p_1x+p_2z,\\
  \dot{z} &=& p_3.
\end{eqnarray*}
Therefore, the information is concentrated on exactly one straight line, which coincides with the limit of the rotation axis for $\eps=0$; see formula~\eqref{axis-xy}. Hence, we define the orbit corresponding to the $\eps=0$ limit of the rotation axis as the {\it singular weak canard in the PWL framework}~\cite{SW01,W05}. 

\subsubsection{Maximal canards and singular strong canard}
Since the reduced system in the central zone does not give any further information, one needs to find another strategy in order to define the singular strong canard. Going back to the system in the central zone for $0<\eps\ll 1$, under natural assumptions we obtain, in Proposition \ref{propcan} below, explicit solutions passing from the attracting slow manifold to the repelling one, which then correspond to maximal canards. In particular, there is exactly one maximal canard that passes from one side to the other without completing a full rotation; by definition, this special solution is the {\it strong canard}; see Figure~\ref{fcr} (b). We then define the {\it singular strong canard in the PWL framework} as the limit of this special solution when $\eps\to0$, while maintaining the central zone open.

The overall system is reversible with respect to the involution $$(x,y,z,t) \mapsto (R(x,y,z),-t),$$ where $R(x,y,z)=(-x,y,-z)$. Consequently, the intersection lines between the slow manifolds and their nearby switching plane are symmetric with respect to $R(x,y,z)$; these lines are given by
\begin{eqnarray}
\label{attrepmandelta}
  \begin{array}{l}
    L^A_{\eps} = \left\{x=-\delta,\;y+\frac{\eps p_2z}{\lambda_A}\;=\;-\delta(1+\lambda_A)-\frac{ p_2p_3}{\lambda^2_A} \eps^2 \right\},\\ \\
    L^R_{\eps} = \left\{x=\delta,\quad y+\frac{\eps p_2z}{\lambda_R}\;=\;-\delta(1-\lambda_R)-\frac{p_2p_3}{\lambda^2_R}\eps^2 \right\}.
  \end{array}
\end{eqnarray}
Let $\left[x \right]$ stand for the integer part of $x$. From the above elements, we conclude the following result. 
\begin{proposition}\label{propcan}
Consider system~\eqref{3d}-\eqref{pwlfunc} with $p_3> 0$, $\delta=\pi\sqrt{\varepsilon}$ and $\varepsilon$ small enough. Assuming that different maximal canards have different flight times, the following statements hold.
\begin{itemize}
 \item [a)] Maximal canards $\gamma$ are reversible orbits, that is the intersection points with the separation planes $\mathbf{p} \in \gamma \cap L^a_{\eps}$ and $\mathbf{q} \in \gamma \cap L^r_{\eps}$ are related to one another by the reversibility $R$; hence $\mathbf{q}=R(\mathbf{p})$.
 \item [b)] If $p_1>0$ and $p_2<0$, for every integer $k$ with $0\leq k \leq [\mu]$, where $\mu$ is the maximal rotation number defined in~\eqref{maxrotnum}, there exists a maximal canard $\gamma_k$ intersecting the switching plane $\{x\!=\!-\delta\}$ at $\mathbf{p}_k=(-\delta,y_{k},z_{k})$ where 
  \begin{equation}\label{def:maximal_canards}
   \begin{array}{l}
    y_{k} = - \left( \left( k+\dfrac 1 2 \right) \dfrac {p_2p_3}{\sqrt{p_1}} + p_1 \right) \pi \varepsilon^{\frac 3 2 } -p_2p_3 \varepsilon^2+ O(\varepsilon^{\frac 5 2}), \\ \\
    z_{k} = -\left(k+ \dfrac 1 2 \right) \dfrac {p_3}{\sqrt{p_1}} {\pi} \sqrt{\varepsilon}+ O(\varepsilon). 
   \end{array}
 \end{equation}
 Moreover, $\gamma_k$ turns $k$ times around the weak canard $\gamma_w$. 
 \item [c)] If $p_1>0$ and $p_2>0$, there exists a unique maximal canard $\gamma_0$ intersecting the switching plane at $\mathbf{p}_0=(-\delta,y_{0},z_{0})$ where the coordinates $y_{0}$ and $z_{0}$ satisfy equation \eqref{def:maximal_canards} with $k=0$.
 \item [d)] If $p_1<0$, there are no maximal canards.
\end{itemize}
\end{proposition}
\begin{proof}
The proof of the proposition is given in Appendix A.
\end{proof}\newline

Therefore, from part (b) and (c), we can conclude that $\gamma_0$ is the strong canard in both the folded-node and the folded-saddle cases, respectively.
\begin{rem}
The assumption about flight time of maximal canards in Proposition~\ref{propcan} implies that every maximal canard is reversible, as it is the case in the smooth context. In this sense we consider this assumption as a very natural one. Note that relaxing this assumption could give the possibility for non-unique maximal canards with a given rotation number, which would then be solutions that are not related to the smooth case and on which we do not want to focus here.
\end{rem}

Due to the reversibility of the maximal canards, the projection of the strong canard onto the $(z,x)$ plane passes through the origin. In the singular limit, only one point  of the strong canard remains in the central zone, that is, the origin itself. However, maintaining that zone open, and noticing that, from \eqref{def:maximal_canards}, the slope $(-z_{0},\delta)$ of the projection of the strong canard converges to $\left(\frac 1 2 \frac {p_3}{\sqrt{p_1}}, 1 \right)$ as $\eps$ tends to zero, we choose to take this direction through the origin as our definition of the \textit{singular strong canard.}

\begin{figure}[!b]
\centering
\includegraphics[width=9cm]{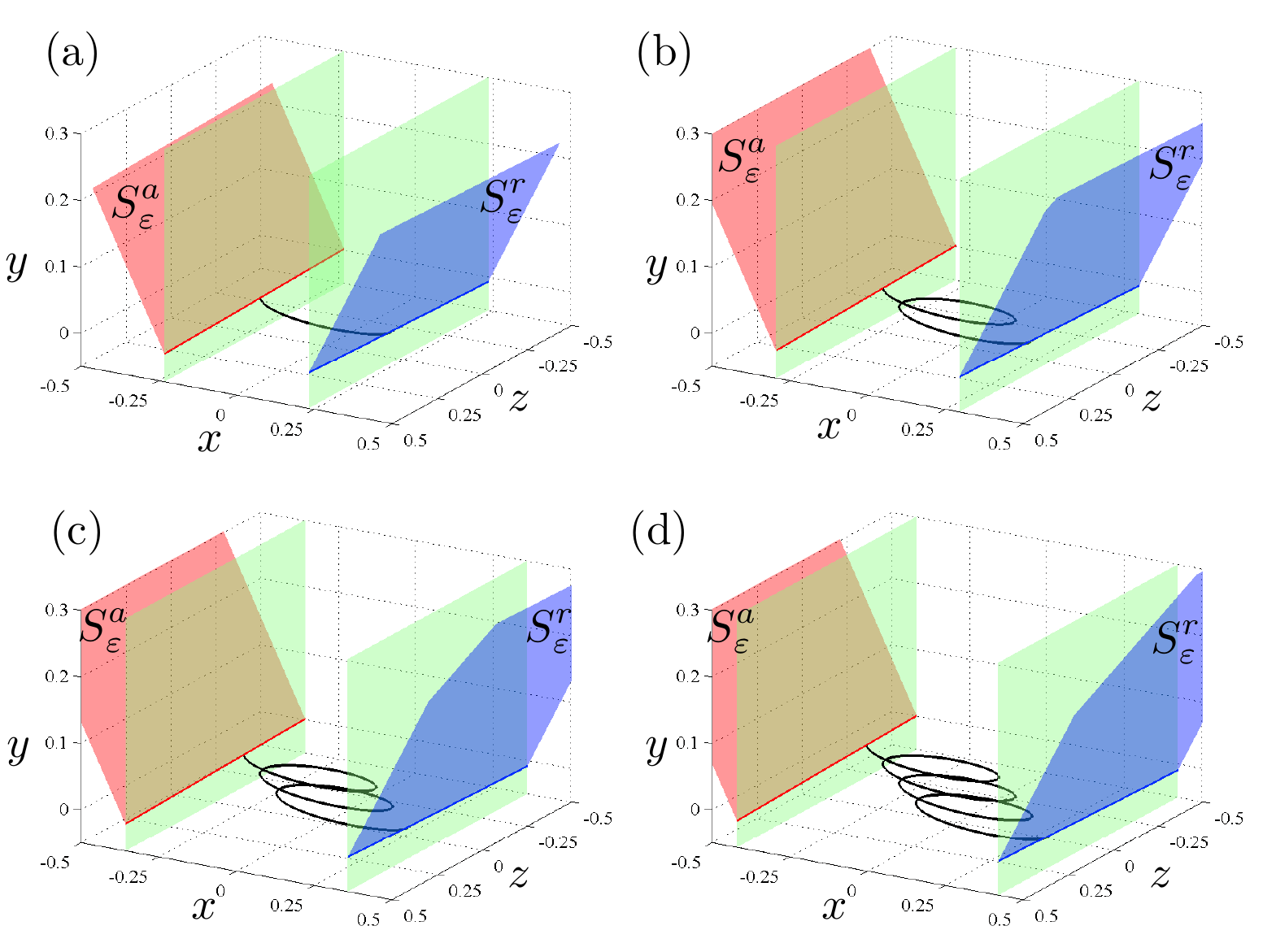}
\caption{\label{excan} Exact canard solutions obtained for $\delta=\delta_k$, $k=0,1,2,3$ in panels (a), (b), (c) and (d), respectively. In each panel, the attracting slow manifold $S^a_{\eps}$ (red) and the repelling slow manifold $S^r_{\eps}$ (blue) are also shown. In each panel, we have fixed: $\eps=0.01$, $p_1=1$, $p_2=-1$, $p_3=0.2$.}
\end{figure}
The singular weak and strong canards that we have just defined in the central zone, are the only two directions that remain, in the singular limit $\eps=0$, as ways to pass from the attracting part to the repelling part of the critical manifold. This is because the central zone, kept open in the limit $\eps=0$, acts as a blow up of the corner line in order to understand the possibilities for such passage from one side of the critical manifold to the other. In the smooth case, this role is played by the folded singularity, which is not an equilibrium of the reduced flow but an ``opening'' through the fold line (along which the reduced system is singular) allowing the passage from one side to the other. Notice that the two directions that allow for this passage in the singular limit of the smooth case are obtained from the linearisation of the desingularised reduced system (DRS), as eigenvectors of the folded singularity (true equilibrium of the DRS); see~\cite{DGKKOW} for 
details. Hence, they play the same role as the singular canards which we have defined in the central zone in the PWL setting. We will identify the singular weak and strong canards in the outer zones, both for the folded saddle and for the folded node cases, in Section~\ref{fsad} and~\ref{fnod}, respectively. 

Although the results presented in Proposition~\ref{propcan} are valid for all $\delta$ small enough (with a proper scaling as a function of $\eps$), we wish to emphasise the fact that the PWL framework also leads to the existence of very simple and explicit maximal canards for certain specific values of $\delta$. These particular values of $\delta$ (within the correct scaling) allow the overall angle in the central zone to be exactly an odd multiple of $\pi$, hence making the expression for one ``selected'' maximal canard very simple. This result is the purpose of the following Proposition and it highlights the simplifying aspect of the PWL framework for canard problems. In the perspectives of using the PWL framework in order to develop simple yet accurate neuron models, the possibility for such ``selected'' explicit canard solutions can potentially be of great use to study the associated boundary between different levels of activity.

\begin{proposition}\label{selected}
Consider system~\eqref{3d}-\eqref{pwlfunc} with $p_1>0$, $p_2<0$, $p_3> 0$, and fix $\delta=\delta_k$ where $\delta_k=-\frac{p_2p_3}{p_1^2}\left(\left(k+\frac{1}{2}\right)\pi\sqrt{\eps p_1}+1\right)$ for $k=0,1,2,\ldots [\mu]$, where $\mu$ is the maximal rotation number defined in~\eqref{maxrotnum}. Then, the $k$th secondary canard (resp. strong canard when $k=0$) crosses the left switching plane $\{x=-\delta_k\}$ at 
$$(x,y,z)=\left(-\delta_k,\eps\frac{p_2p_3}{p_1},-\left(k+\frac{1}{2}\right)\pi\frac{p_3}{p_1}\sqrt{\eps p_1}\right),$$ 
and takes, in the central zone, the following explicit time parametrisation:
\begin{eqnarray}
\label{k-x} x_{\mathrm{sc}}(t) &=& \frac{p_2p_3}{p_1^2}\cos(\sqrt{\eps p_1}t)-\sqrt{\eps p_1}\frac{p_2p_3}{p_1^2}\left(\sqrt{\eps p_1}t - \left(k+\frac{1}{2}\right)\pi\right),\\
\label{k-y} y_{\mathrm{sc}}(t) &=& -\sqrt{\eps p_1}\frac{p_2p_3}{p_1^2}\big(\sin(\sqrt{\eps p_1}t)-\sqrt{\eps p_1}\big),\\
\label{k-z} z_{\mathrm{sc}}(t) &=& ~~\sqrt{\eps}p_3\left(\sqrt{\eps}t - \left(k+\frac{1}{2\sqrt{p_1}}\right)\pi\right), 
\end{eqnarray}
for $0\leq t \leq (2k+1)\pi \frac 1{\sqrt{\varepsilon p_1}}$.
\end{proposition}
\begin{proof}
The proof of the proposition is given in Appendix B
\end{proof}\newline

\noindent\textbf{Remark.} The quantity $-\frac{p_2p_3}{p_1^2}$ plays the role of the parameter $\nu$ used to measure the speed of the slow drift in minimal models for folded singularities; see~\cite{DGKKOW}.

Figure~\ref{excan} proposes an illustration of Proposition~\ref{selected} for $k=0$ to $3$ in panels (a), (b), (c) and (d), respectively.

\subsection{Singular flow far away from the origin}
We continue this section by a remark pushing further the comparison between the smooth version and the PWL version of the minimal system~\eqref{3d}, in the singular limit. The singular flow on the attracting part and on the repelling part of the critical manifold, away from a vicinity of the origin, are topologically equivalent in the smooth case and in the PWL case. Indeed, the slow flow of system~\eqref{3d} with $f(x)=x^2$ can be understood from the reduced equation
\begin{eqnarray}
\label{drssmoothxy}
  \begin{array}{l}
      -2x\dot{x} = -(p_1x+p_2z),\\
      \quad\;\;\;\dot{z} = p_3,
  \end{array}
\end{eqnarray}
after performing the usual desingularisation, which amounts to a time rescaling by a factor $-2x$ in~\eqref{drssmoothxy}. One then studies the (regular) DRS and maps its dynamics back to the original reduced system with a necessary reversal of time orientation along trajectories on the repelling part of the critical manifold, that is, in the $\{x>0\}$ half-space. However, one can also consider two systems defined separately on the two half-spaces $\{x<0\}$ and $\{x>0\}$. In this way, the following reduced equation
\begin{eqnarray}
\label{redsmoothxy}
  \begin{array}{l}
    \dot{x} = \sgn(x)(p_1x+p_2z),\\
    \dot{z} = 2p_3|x|,
  \end{array}
\end{eqnarray}
is obtained. Therefore, one can expects that, outside a neighbourhood of the origin, the phase portraits of systems~\eqref{drspwlxy} and~\eqref{redsmoothxy} are topologically equivalent. Simple computations will illustrate this point in the next sections with Figures~\ref{fsadfig1} to~\ref{fsadnodIIfig}.

\subsection{Folded singularities in the PWL setting}
\label{foldsing}
\begin{figure}[!t]
\centering
\includegraphics[width=10cm]{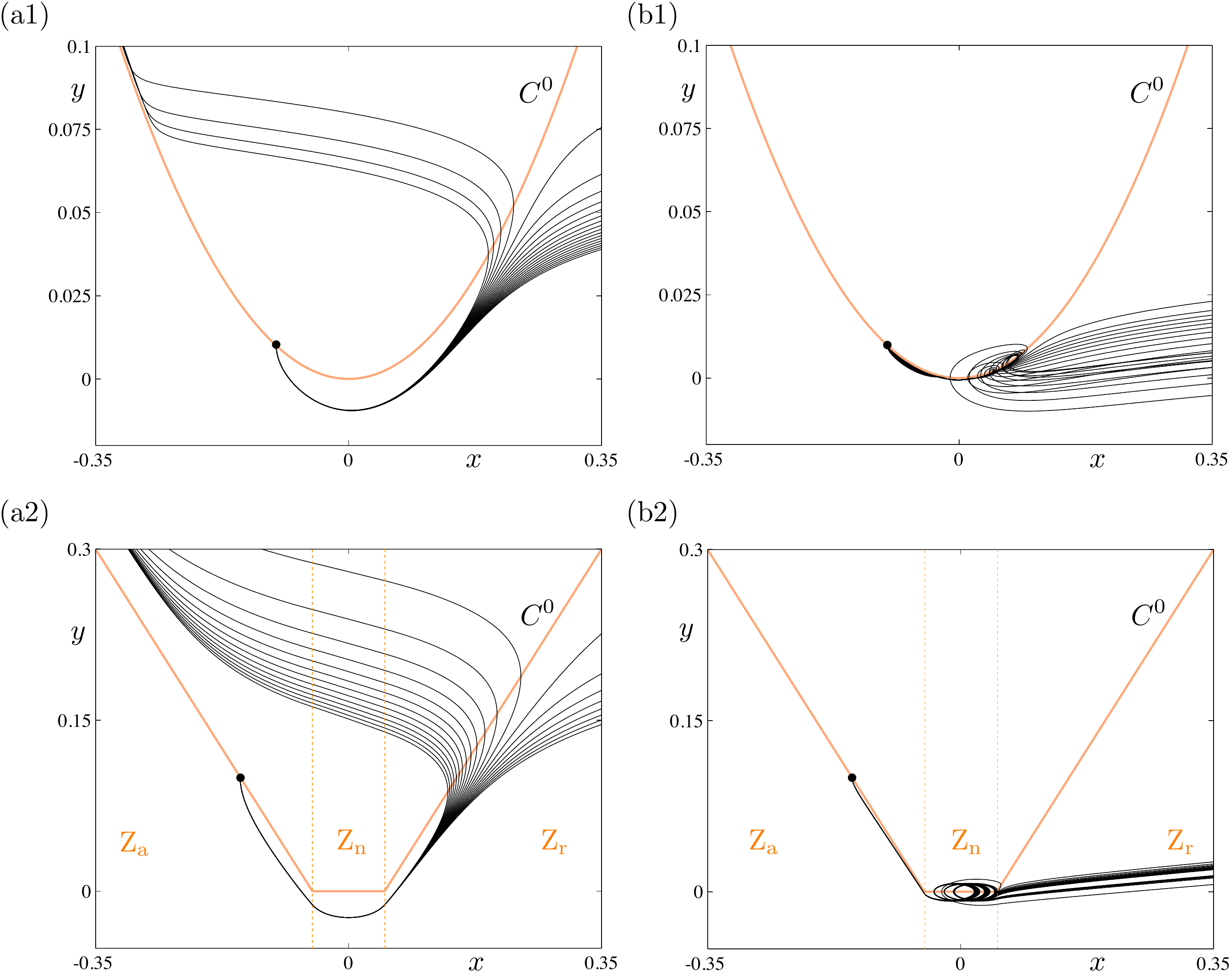}
\caption{\label{fsfn}Canards near folded singularities: folded saddles on panels (ai), i=1,2, and folded nodes on panels (bi). The top panels correspond to system~\eqref{3d} with $f(x)=x^2$, and the bottom panels correspond to the same system with a 3-piece symmetric piecewise-linear function for $f=f_{\delta}$ given in \eqref{pwlfunc}. In each panel are shown 20 trajectories where only the $z$-coordinate of the initial condition is varied; a black dot marks the $(x,y)$-coordinates, which are constant across this sweep. Also plotted is the critical manifold $C^0$ as well as switching planes in the PWL cases, that is, on the bottom panels.}
\end{figure}
The singular canards that we have defined as special directions in the limit $\eps=0$ naturally cross at the origin. We then define this point as the {\it folded singularity for the PWL framework}. In order to classify PWL folded singularities and compare with the smooth cases (folded saddle, folded node, etc..), we first recall the conditions on parameters $p_i,\;i=1,..,3$~\cite{P08} which characterize the different folded singularities in smooth slow-fast systems written in the minimal form~\eqref{3d}:
\begin{itemize}
\item[$\circ$] the {\bf folded-saddle} case corresponds to $p_2p_3>0$;
\item[$\circ$] the {\bf folded-node} case corresponds to $p_1>0$ and $-p_1^2/8<p_2p_3<0$;
\item[$\circ$] the {\bf folded-focus} case corresponds to $p_1>0$ and $p_2p_3<-p_1^2/8$;
\item[$\circ$] the {\bf folded saddle-node} case corresponds to $p_2p_3=0$, with two sub-cases : {\bf type I} ($p_2=0,\;p_3\neq0$) and {\bf type II} ($p_2\neq0,\;p_3=0$).
\end{itemize}
Without loss of generality, we will only consider the case $p_3\geq0$ since similar results are obtained for $p_3\leq0$ by reversing the sign of $p_2$; see~\eqref{drspwlxy}. Besides, we will not consider the folded focus case because it does not create any canards for $\eps>0$~\cite{SW01}. Furthermore, we will not consider the folded node case with $p_1<0$ since only {\it faux canards}~\cite{BCDD81} (that is, trajectories flowing from the repelling side to the attracting side) exist in this case; in particular no MMO dynamics can be obtained near an unstable folded-node singularity. Finally, in the folded saddle case, we will consider both $p_1>0$ and $p_1<0$ since in these cases  the role of the true canard and that of the faux canard are exchanged.

Before engaging into the analysis of the different types of PWL folded singularities, we first verify numerically if the criteria on the parameter values are the same in the PWL setting and in the smooth setting. Taking one parameter set that corresponds to the folded saddle case ($p_1=1$, $p_2=1$, $p_3=0.1$) and one that corresponds to the folded node case ($p_1=1$, $p_2=-1$, $p_3=0.1$), direct simulations show that the observed dynamics in the PWL system and in the smooth system are entirely comparable. The results of the simulations are presented in Figure~\ref{fsfn}. It clearly shows that the qualitative behaviour of system~\eqref{3d} is the same in the PWL setting (with $f=f_{\delta}$ given by \eqref{pwlfunc}) and in the smooth setting (with $f=x^2$). Furthermore, criteria on the main parameters in order to obtain folded node type dynamics or folded saddle type dynamics also seem to be the same in both frameworks. Indeed, the folded-saddle condition for $p_1>0$ imposes the fact that the flow along the rotation axis goes in the opposite direction (see~\eqref{axis-xy}); when $p_1<0$, there is no rotation. The folded-node condition follows from the fact that $\mu$ has to be greater than $1$ in that case (see~\eqref{maxrotnum}). Up to a positive scaling factor equal to $\sqrt{p_1}/8$, this gives the same conditions as in the smooth case. Namely, 
\begin{itemize}
\item[$\circ$] the {\bf folded-saddle} case corresponds to $p_2p_3>0$;
\item[$\circ$] the {\bf folded-node} case corresponds to $p_1>0$ and $-p_1\sqrt{p_1}<p_2p_3<0$;
\item[$\circ$] the {\bf folded-focus} case corresponds to $p_1>0$ and $p_2p_3<-p_1\sqrt{p_1}$;
\item[$\circ$] the {\bf folded saddle-node} case corresponds to $p_2p_3=0$, with two sub-cases : {\bf type I} ($p_2=0,\;p_3\neq0$) and {\bf type II} ($p_2\neq0,\;p_3=0$).
\end{itemize}
We will now make this splitting on parameter sets and verify that the singular phase portraits in the outer zones are compatible in each framework.
\begin{figure}[!b]
\centering
\includegraphics[width=10cm]{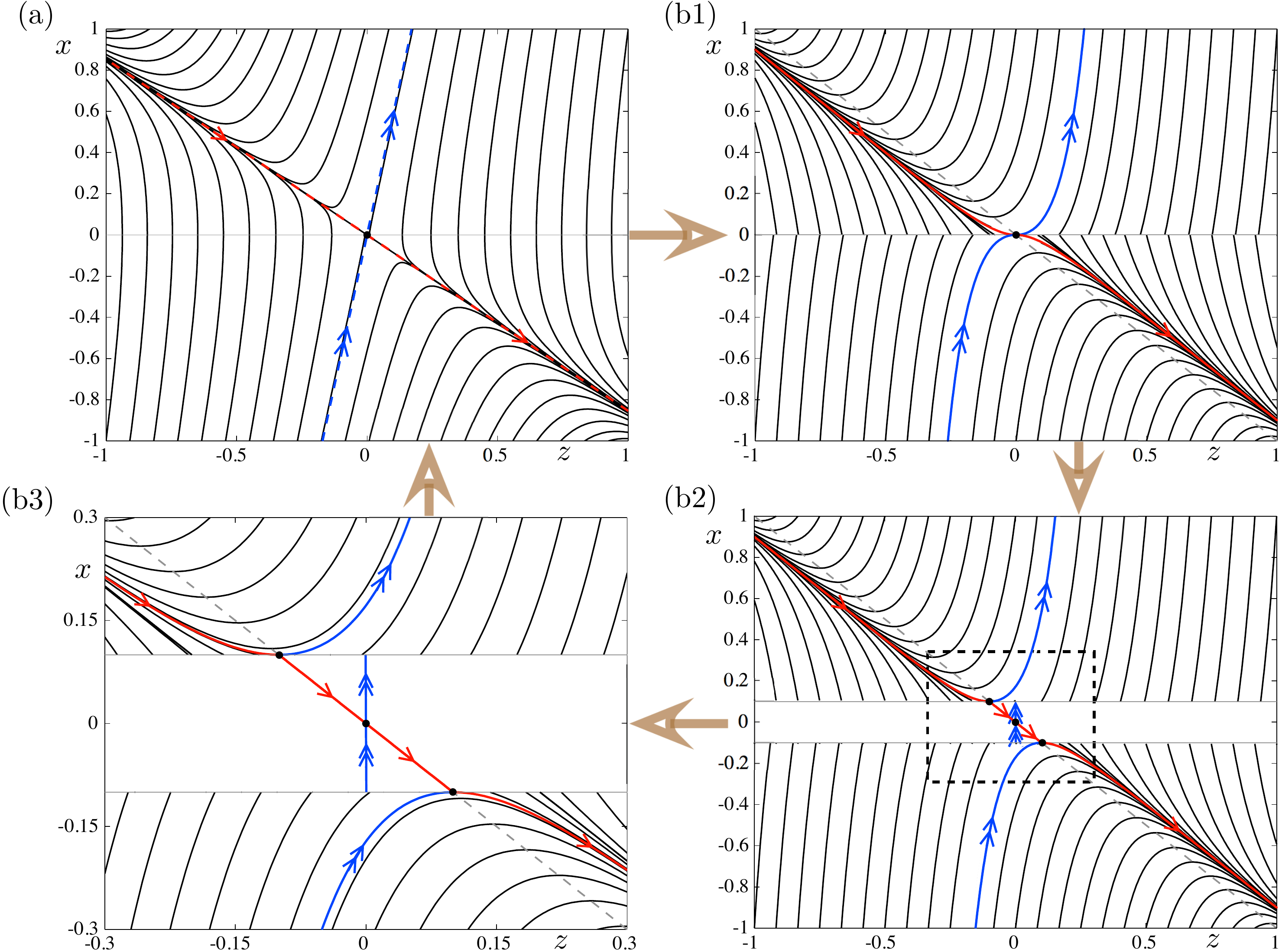}
\caption{\label{fsadfig1}Slow flow near a PWL folded saddle (black dot), for the following parameter values: $p_1=1$, $p_2=1$, $p_3=0.1$. The singular phase portrait of the smooth system for the same parameter values is shown in panel (a). In panel (b1), the slow flow is presented, in its two-zonal configuration. In order to reveal the singular weak canard, one performs an opening with the $\eps=0$ limit of the central zone considered in that blown-up zone; the result is shown panels (b2) and (b3) (zoom). In all panels, the red and blue lines/curves denote the singular faux canard and the singular strong canard, respectively; also, in panels (b1)--(b3), the dashed grey line denotes the $z$-nullcline.}
\end{figure}

\subsubsection{Folded saddle: singular phase portrait}
\label{fsad}
\begin{figure}[!b]
\centering
\includegraphics[width=10cm]{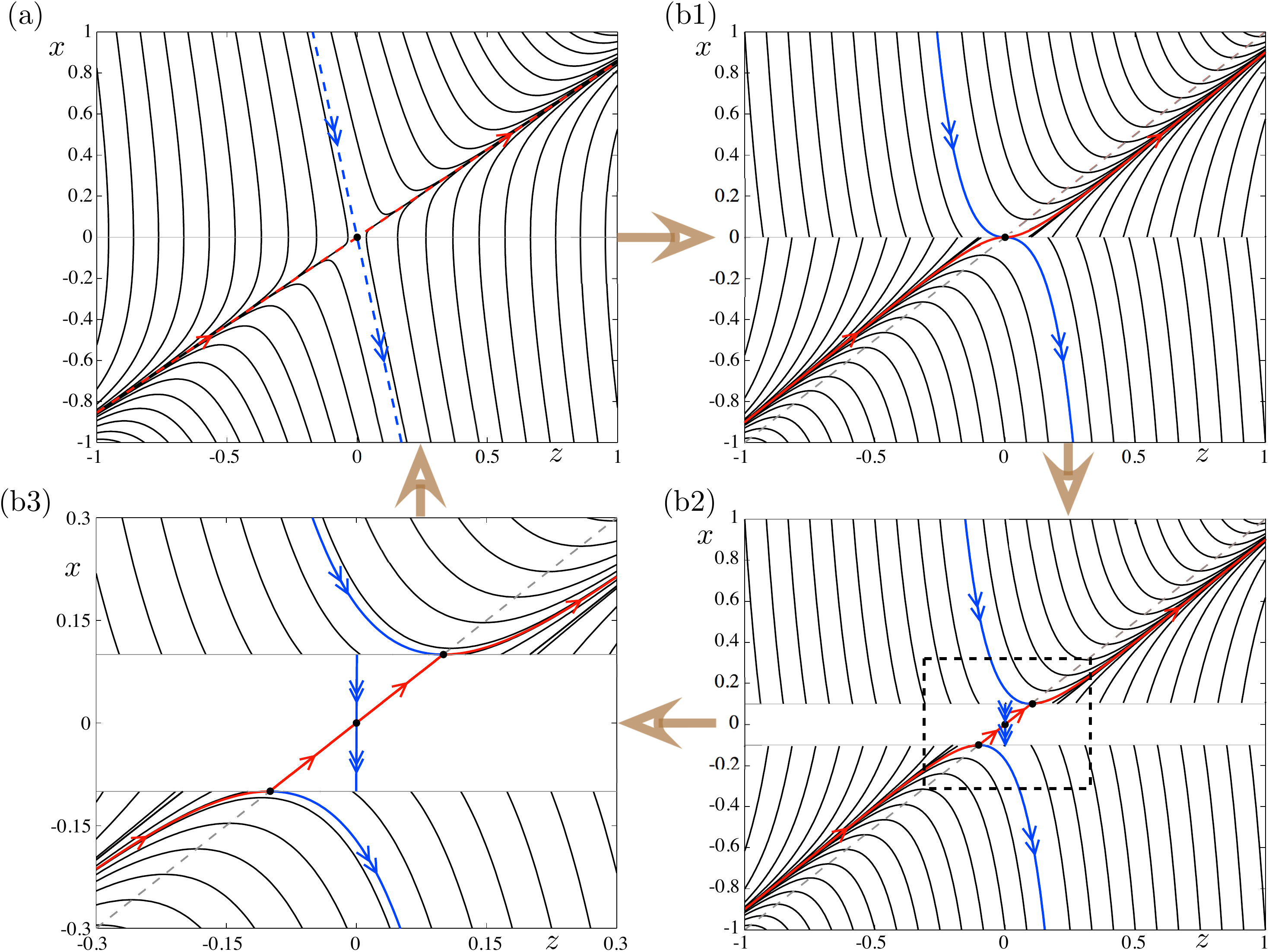}
\caption{\label{fsadfig2}Same figure as~\ref{fsadfig1} but with $p_1=-1$. The roles of the singular true canard and of the singular faux canard are now exchanged.}
\end{figure}
\begin{figure}[!b]
\centering
\includegraphics[width=10cm]{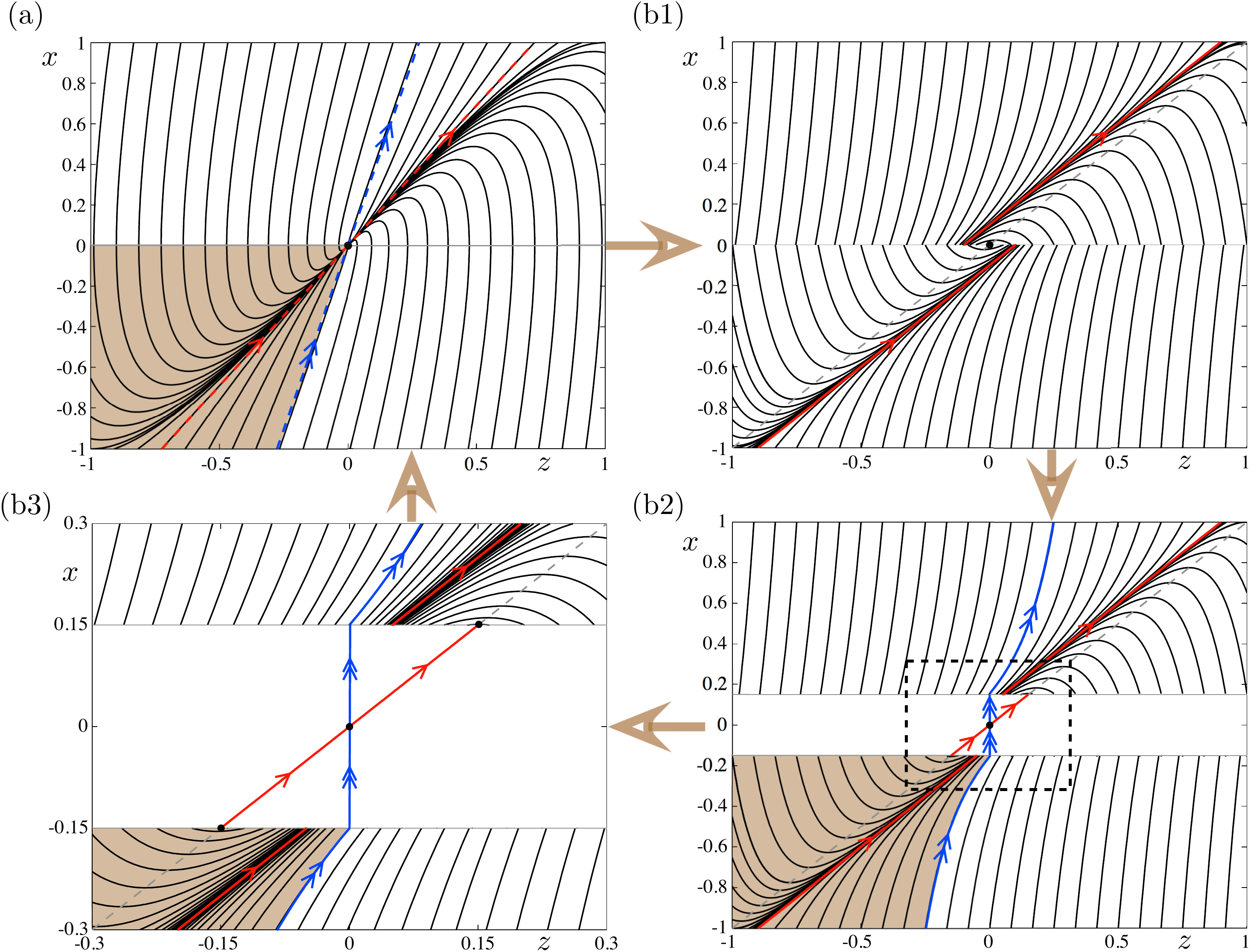}
\caption{\label{fnodfig}Slow flow near a PWL folded node (black dot), for the following parameter values: $p_1=1$, $p_2=-1$, $p_3=0.1$. The singular phase portrait of the smooth system for the same parameter values is shown in panel (a). In panel (b1), the slow flow is presented, in a two-zonal configuration. In order to reveal the singular weak canard, one performs an opening with the $\eps=0$ limit of the central zone considered in that blown-up zone; the result is shown panels (b2) and (b3) (zoom). In all panels, the red and blue lines/curves denote the singular weak canard and the singular strong canard, respectively; also, in panels (b1)--(b3), the dashed grey line denotes the $z$-nullcline.}
\end{figure}

Consider system~\eqref{drspwlxy} with parameters satisfying $p_2p_3>0$. It is easy to check that the half straight line given by 
\begin{eqnarray}\label{hsl1}
p_1x+p_2z &=& -\frac{p_2p_3}{p_1}
\end{eqnarray}
is invariant under the flow in the $\{x>0\}$ half-plane. Similarly, the half straight line given by 
\begin{eqnarray}\label{hsl2}
p_1x+p_2z &=& \frac{p_2p_3}{p_1}
\end{eqnarray}
is invariant under the flow in the $\{x<0\}$ half-plane. Furthermore, the $x$-nullcline, $p_1x+p_2z=0$, is the locus of tangency points in the direction of $z$. Therefore, in the two-zonal system at $\eps=0$ (that is, without the opening), the origin is a so-called {\it visible-visible tangency point}~\cite{JH11} of the flow; see panel (b1) in Figure~\ref{fsadfig1}. The direction of the flow is easily deduced from the $z$ equation in~\eqref{drspwlxy}. This panel shows that the behaviour of the slow flow is equivalent, outside a neighbourhood of the origin, to that of the smooth system, presented on panel (a). 

However, this description of the singular flow does not provide any information about the crossing from one zone to the other. As we explained in the previous section, such crossing through singular canards can only be revealed by maintaining the central zone open in the singular limit, which corresponds to performing a blow-up. This behaviour is shown in panels (b2) and (b3), the latter being a magnified view of the former near the origin. The red and the blue orbits in the central zone are the singular weak and strong canards defined in the previous section. The blow-up splits the double tangency point at the origin into two tangency points that lie along the $x$-nullcline, which corresponds to the direction of the singular weak canard in the central zone. The resulting tangency point in each outer zone determines a separatrix that gives one side of each singular canard in that zone. In the central zone, the singular weak and the singular strong canards are the blown-up images of the two separatrices meeting at the origin in the two-zonal system (panel (b1)). Note that after the blow-up is performed, the resulting singular weak canard is connected across all three zones whereas the singular strong canard is disconnected. However, both objects converge in the two-zonal limit to the two separatrices which are tangent to the corner line at the origin. In fact, this process corresponds to a concatenation of the real singular flow together with singular directions obtained in the central part after blow-up, which are necessary to identify the canards. In this way, we are able to compare the singular phase portrait near the PWL folded singularity (panel (b3)) with that of the smooth case (panel (a)).

The only difference between Figure~\ref{fsadfig1} and Figure~\ref{fsadfig2} is the sign of $p_1$, positive in the former and negative in the latter. The role of the sign of $p_1$ is to determine which singular canard is a true canard and which one is a faux canard. In particular, when $p_1$ is negative, the singular strong canard happens to be a faux canard. This fact also happens in the smooth context although it is not often commented about in the literature. However, the identification of the singular canards does not depend on the sign of $p_1$.

\begin{figure}[!b]
\centering
\includegraphics[width=10cm]{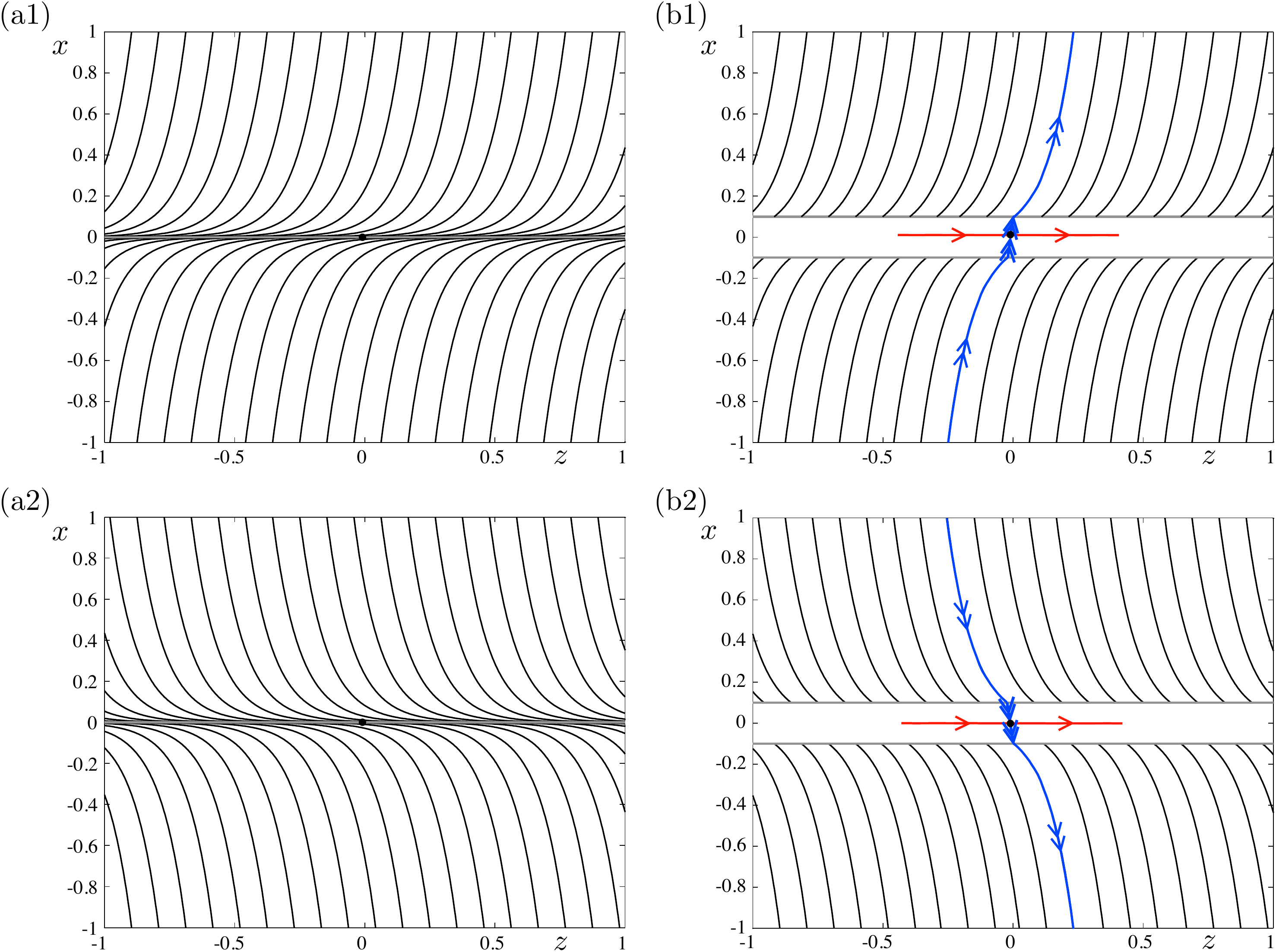}
\caption{\label{fsadnodIfig}Slow flow near a PWL folded saddle-node, corresponding to $p_2p_3=0$. The case shown here is in fact the folded saddle-node of type I (FSNI), that is, $p_2=0$: (a) with $p_1>0$ and (b) with $p_1<0$. The slow flow is shown. In the folded saddle-node limit, the weak canard disappears as a canard, that is, as a way to pass from one side to the other.}
\end{figure}
\subsubsection{Folded node: singular phase portrait}
\label{fnod}
The singular flow in the folded node case, that is, when $p_2p_3<0$, is obtained in a very similar manner as in the folded saddle case. In particular, the two half-straight lines~\eqref{hsl1} and~\eqref{hsl2} are still invariant under the flow in the respective zone. The locus of tangency points also persists in this case, however the origin now corresponds to an invisible-invisible tangency, usually referred to as a {\it Teixeira singularity}~\cite{JH11,PTV14}; see Figure~\ref{fnodfig}. Here again, the singular phase portrait in the two-zonal configuration (panel (b1)) does not allow to find the possibilities of crossing from one zone to the other. After performing the central blow-up, we find the two directions that realise this crossing, corresponding to the singular weak and strong canards. In order to determine their extensions to the outer zone, we also use the role of singular canards as separatrices. Indeed, the singular weak canard separates trajectories that go towards the fold line, 
either in forward time or in backward time, from those that go to infinity. On the other hand, the singular strong canard separates trajectories that cross to the other side from those that do not; it defines one boundary of the {\it funnel} (colored area in Figure~\ref{fnodfig}). See panel (a) for details about the smooth folded node case. Consequently, we define the singular weak and strong canards in the outer zones in the same way. We remark that the central blow-up in this case is more subtle than in the folded saddle case, in the following sense : it separates not only the double tangency point at the origin into two tangency points, but also the intersection points between the singular canards and the blown-up region. We can then identify the funnel region in the PWL setting; see panels (b2) and (b3). 

Note that in the PWL folded node scenario, the singular strong canard is connected across all three zones (which is consistent with its role of boundary of the funnel region), whereas the singular weak canard is not. The situation is opposite to that of the PWL folded saddle case.

\subsubsection{Limit cases: folded saddle-node transitions}
\label{fsadnod}
The singular phase portraits of the folded saddle-node of type I (FSNI), that is, for $p_2=0$, are presented in Figure~\ref{fsadnodIfig}: the top panels display the case $p_1>0$ and bottom panels the case $p_1<0$. We can immediately notice that these two phase portraits are compatible with those on Figure~\ref{fsadfig1} and~\ref{fsadfig2} in the limit $p_2=0$, respectively. Indeed, in that limit, each tangency point on the boundary of the blown-up region converges to infinity. It is easy to see from equations~\eqref{axis-xy} that the axis of rotation at $\eps=0$ converges to the horizontal axis $x=0$ in the limit $p_2=0$. Then, this segment does not connect one side of the critical manifold to the other, which shows that the singular weak canard does not exist as a canard in the FSNI.
\begin{figure}[!b]
\centering
\includegraphics[width=12cm]{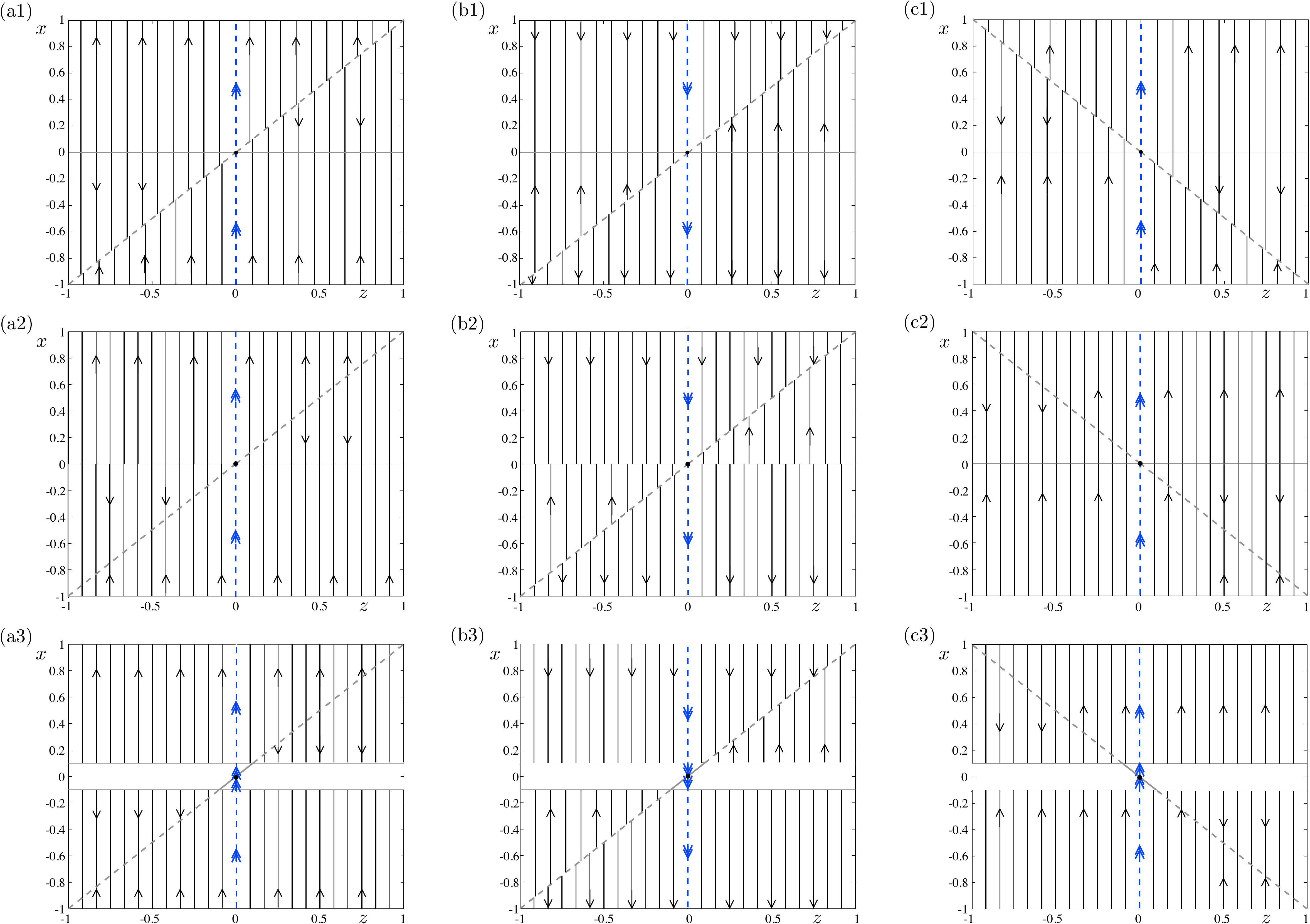}
\caption{\label{fsadnodIIfig}Slow flow near a PWL folded saddle-node, corresponding to $p_2p_3=0$; the case shown here is in fact the folded saddle-node of type II, that is, $p_3=0$. Top panels show the smooth singular phase portraits; middle panels show the PWL singular phase portraits; bottom panels show the PWL singular phase portraits with the opening performed. However, no information can be obtained from that opening since $z=constant$. Hence, no connection is possible between the attracting zone and the repelling zone. This is consistent with the fact that the weak canard disappears in the folded saddle-node limit. In all panels, the blue lines correspond to the singular strong canard. The dashed grey line corresponds to the $z$-nullcline, which in this case is a line of equilibria.}
\end{figure}
This fact is well-known in the smooth context~\cite{krupa10}. Regarding the singular strong canard, we can construct it in a similar fashion as in the 
previous cases of folded saddle and folded node. Namely, 
performing the 
central blow-up and using formula~\eqref{def:maximal_canards} allows to define the direction corresponding to the singular strong canard in this zone: $(\frac 1 2\frac{p_3}{\sqrt{p_1}},1)$. This particular direction, going through the origin, hits the two boundaries of the blown-up region in one point, from which passes a unique trajectory solution to the system in the corresponding outer zone; see panels (b1) and (b2) of Figure~\ref{fsadnodIfig}.\newline
Note that the phase portrait of the FSNI in the smooth case cannot be completely obtained from the canonical form; see the remark made at the end of Section 3.1.1 of~\cite{P08}. Indeed, one needs to incorporate higher-order terms in one slow equation of the canonical form; for a detailed analysis of this case, see~\cite{krupa10}. Therefore, we need to extend accordingly the PWL canonical form in order to obtain the complete singular phase portrait near a FSNI. Presumably, it would be convenient to include an absolute-value term also in this slow equation. This is beyond the scope of the present paper and will be the subject of future work. 

In the case of folded saddle-node of type II (FSNII), that is, for $p_3=0$, according to the slow flow given by equations~\eqref{drspwlxy}, the dynamics is given by a linear flow for the variable $x$ along the lines $\{z=\mathrm{constant}\}$. Furthermore, $\{x=-\frac{p2}{p1}z\}$ is now a line of equilibria of the slow flow, the stability of these equilibria depending on $z$ and $p_1$; see Figure~\ref{fsadnodIIfig} panels (a1)-(a3) for the case $p_1>0,\;p_2<0$, panels (b1)-(b3) for the case $p_1<0,\;p_2>0$, and panels (c1)-(c3) for the case $p_1>0,\;p_2>0$. The top panels show the singular phase portraits for the smooth system; the central and bottom panels show the singular phase portraits for the PWL system without and with the opening, respectively. Similar to the smooth case, we obtain that the singular weak canard disappears as a direction of travel from one side to the other, given that all dynamics disappear in the central zone. However, it persists as a line of equilibria, which also 
corresponds to the limit for $\eps=0$ of the axis of rotation. The singular strong canard persists in the same way as in the smooth case, that is, not as a trajectory that passes from one side to the other but as a trajectory that converges to or diverges from a bifurcation point (the FSNII) depending on the zone; see panels (a1)-(c1) of Figure~\ref{fsadnodIIfig}.

Consequently, we have now shown that for all cases giving rise to persistent canards (i.e., all cases except folded focus), we can construct singular phase portraits and define singular canards across all three zones, in a way that is entirely compatible with the smooth case.

\subsection{Canards for $\boldsymbol{\eps>0}$}
We now gather informations about the canard solutions that perturb from the singular canards studied in the previous section. We will first focus on canards near a folded node since this is the most relevant case to study mixed-mode oscillations. Then, we will reveal a novel aspect of canard dynamics near a folded saddle. 

\begin{figure}[!b]
\centering
\includegraphics[width=11cm]{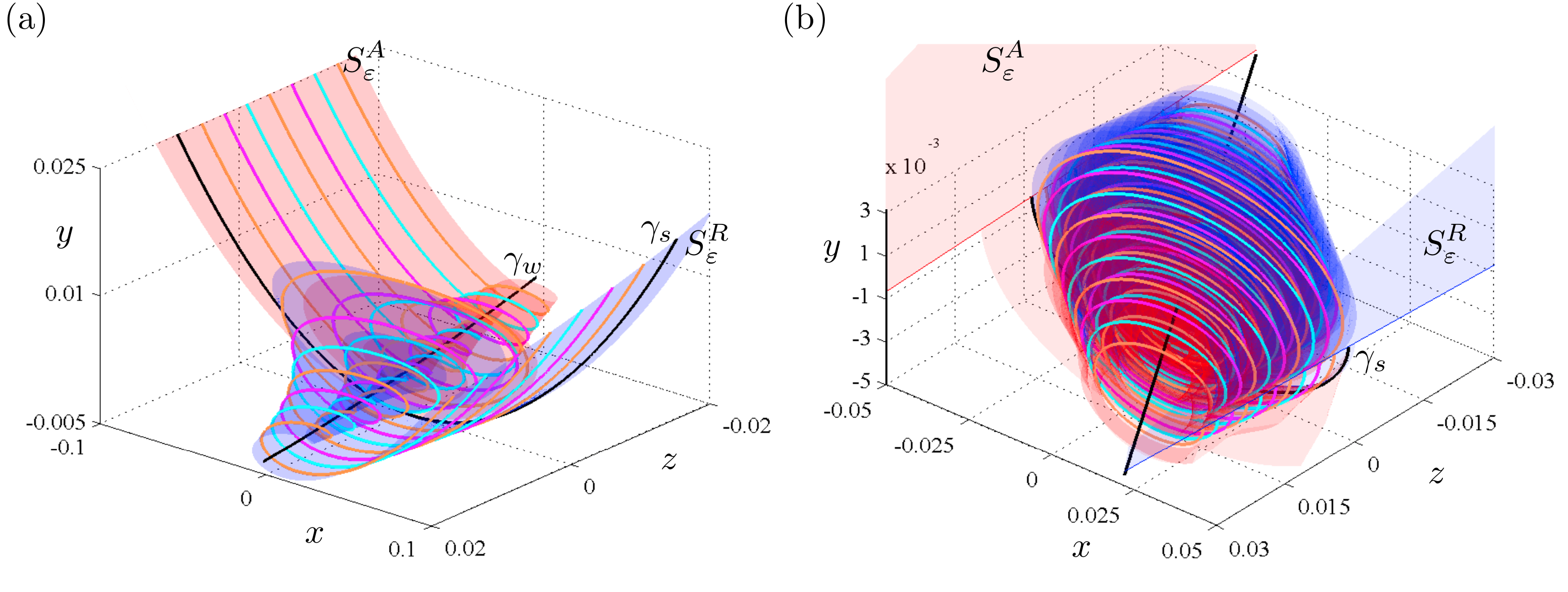}
\caption{\label{slowmancan}Attracting slow manifold $S^a_{\eps}$ (red) and repelling slow manifold $S^r_{\eps}$ (blue), and secondary canards (coloured orbits) near a folded node. Panel (a) shows the smooth case, and panel (b) shows the PWL case. Also shown are the strong canard $\gamma_s$ and the weak canard $\gamma_w$ (not shown on panel (b) since it is not a maximal canard in the PWL, simply the axis of rotation).}
\end{figure}
\subsubsection{Primary and secondary maximal canards near a folded node}
\label{seccan}
All the orbits $\gamma_k$ defined in Proposition \ref{propcan} correspond to maximal canards, i.e. orbits flowing from the attracting slow manifold to the repelling one. The value of $k$ for these solutions is directly related to the number of rotations that they complete in the central zone, and this number is bounded by $\mu$, see~\eqref{maxrotnum}. A theoretical analysis about the bifurcation of the maximal canards near a PWL folded node is an on-going work and will be 
the subject of a follow-up paper. Figure~\ref{slowmancan} is a very clear illustration of the great degree of similarity between the dynamics near a folded node (slow manifolds, canards) in the smooth case (panel (a)) and in the PWL case (panel (b)). However, the PWL framework offers a surprising result, namely the fact that the weak canard, defined as the axis of rotation of the system in the central zone for $\eps>0$, is not a maximal canard. The reason is simply that the axis of rotation connects to the repelling side at a distance of $O(\eps)$ from the repelling slow manifold. This is easy to check using the equations~\eqref{axis-xy} of the axis of rotation, and the equations~\eqref{attrepmandelta} of the intersection lines of the slow manifolds with the switching planes. In order to be a maximal canard, the axis of rotation would need to exactly connect to the repelling slow manifold. We do not see this unexpected result as a problem, but rather as a question to be revisited in the smooth theory.

\subsubsection{SAOs near a folded saddle}
\label{fsadsaos}
\begin{figure}[!t]
\centering
\includegraphics[width=10cm]{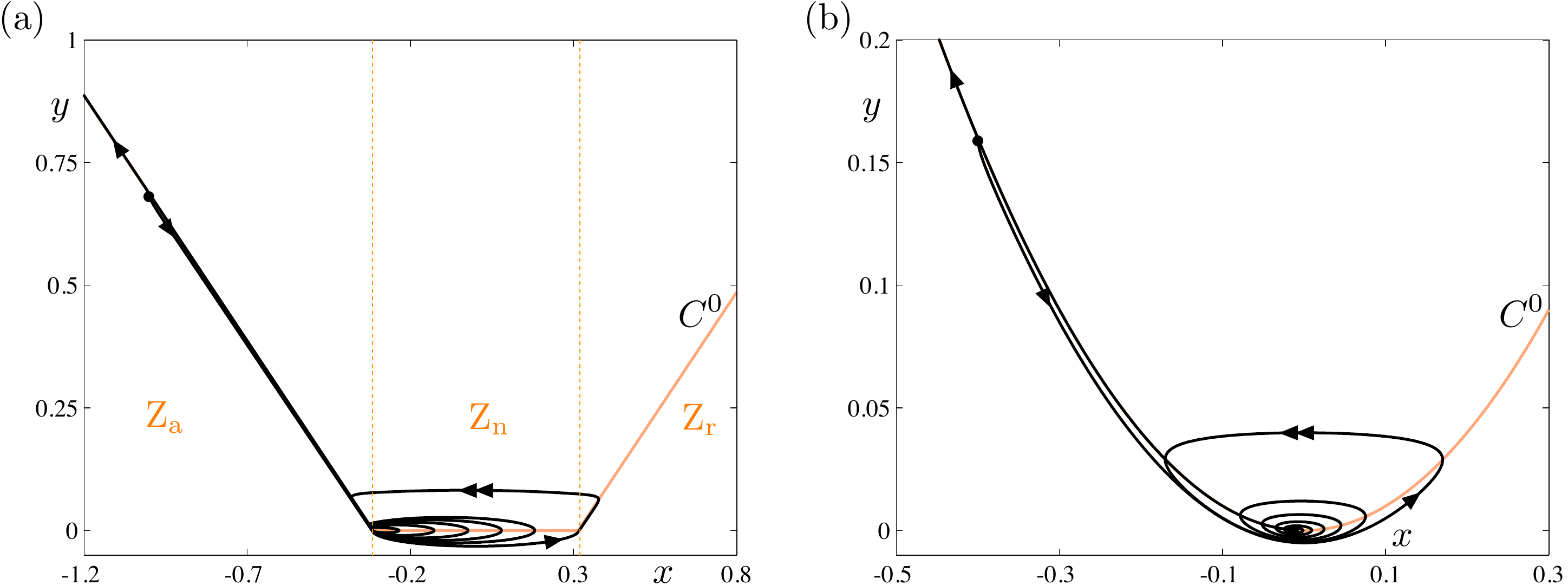}
\caption{\label{fsadsaosfig}Transient small oscillations near a folded saddle; PWL case in panel (a), smooth case in panel (b).}
\end{figure}
Our strategy to construct PWL slow-fast systems with folded singularities, and associated canard solutions, is very much inspired from that of Arima {\it et al.}~\cite{ARIMA97}, that is, it relies on introducing a central zone in between the attracting and the repelling sides of the critical manifold. Furthermore, we impose that the linear dynamics in that central zone has purely imaginary eigenvalues in order to ensure spiralling motion in this zone. Now, it is immediate to see that the structure of the central dynamics, where the phase space is foliated by invariant cylinders and every solution spirals around a common axis of rotation, does not depend on the values of the parameters $p_1, p_2$ and $p_3$; however, $p_1$ must be positive for rotation to occur, as mentioned in Section~\ref{genfsing}. In particular, one can expect to find trajectories displaying small oscillations also near a PWL folded saddle and not only near a PWL folded node; see Figure~\ref{fsadsaosfig} (a) for an example of 
such a trajectory near a PWL folded saddle. This phenomenon could at first appear as a shortcoming of the PWL framework, but in fact it is the exact 
opposite: it is a general fact, which also exists in the smooth case and can be naturally understood using the PWL setting. As an illustration, we show in Figure~\ref{fsadsaosfig} (b) an example of trajectory displaying SAOs near a smooth folded saddle. To the best of our knowledge, this fact is not yet discussed in smooth literature and it is being currently investigated in an independent manuscript in preparation~\cite{mwinprep}. This gives a good example of how the PWL framework not only reproduces all the richness of the dynamics present in the smooth case, but can also provide new information about the smooth case, at least suggest revisiting it. Note that SAOs near folded saddles are obviously less interesting than SAOs near folded nodes, given that one would need a non-smooth global return mechanism (reset, discontinuous jump, etc..) in order to construct MMOs using such trajectories.

\section{Three-dimensional PWL slow-fast systems: global return and robust MMOs}
\label{pwlmmo}

In the previous section we have analysed the local dynamics, that is, we have identified (using the singular blow-up) the folded singularities in the minimal three-dimensional PWL slow-fast systems~\eqref{3d} with the function $f$ given by~\eqref{pwlfunc}. The natural next step is to verify that we can map entirely also the global dynamics that occurs in the smooth case with that from the PWL case. Indeed, our long-term goal is to construct and analyse PWL models displaying robust canard-induced MMOs. In this paper, we only want to showcase that one can easily construct a minimal PWL model displaying MMOs near a folded node.
The way to achieve this is in two steps. First, adding a fourth zone to the minimal system~\eqref{3d}-\eqref{pwlfunc}, in order to allow for large-amplitude oscillations; this corresponds to considering the PWL function $\widetilde{f}_{\delta}$ defined by
\begin{equation}\label{pwlfunc4min}
   \widetilde{f}_{\delta}(x) = \left\{
     \begin{array}{ll}
       -x-\delta  & \mathrm{if}\quad  x  \leq -\delta,\\
        \;\;\;0    & \mathrm{if}\quad |x| \leq  \delta,\\
        \;\;\;x-\delta  & \mathrm{if}\quad  \delta<x<x_0,\\
        -x+2x_0-\delta & \mathrm{if}\quad  x  \geq x_0.
     \end{array}
   \right.
\end{equation}
Then, we just need to add linear terms to the $z$ equation in order to obtain a global return mechanism. The rationale behind this construction is very much inspired from the construction of the four-dimensional model presented in~\cite{desroches13} and displaying Mixed-Mode Bursting Oscillations (MMBOs). The resulting model is then simply formed by equations~\eqref{3d}-\eqref{pwlfunc4min} where we append to the $z$ equation linear terms in $x$, $y$ and $z$ in order to create a global return; its equations are given by:
\begin{eqnarray}
\label{pwlgloret-xyz}
  \begin{array}{l}
    \eps\dot{x} = - y + \widetilde{f}_{\delta}(x)\\
    \;\;\dot{y} =  p_1x + p_2z\\ 
    \;\;\dot{z} =  p_3+\alpha_1(x-\kappa)+\alpha_2(y-\zeta)+\alpha_3(z-\xi).
  \end{array}
\end{eqnarray}
\begin{figure}[!t]
\centering
\includegraphics[width=11cm]{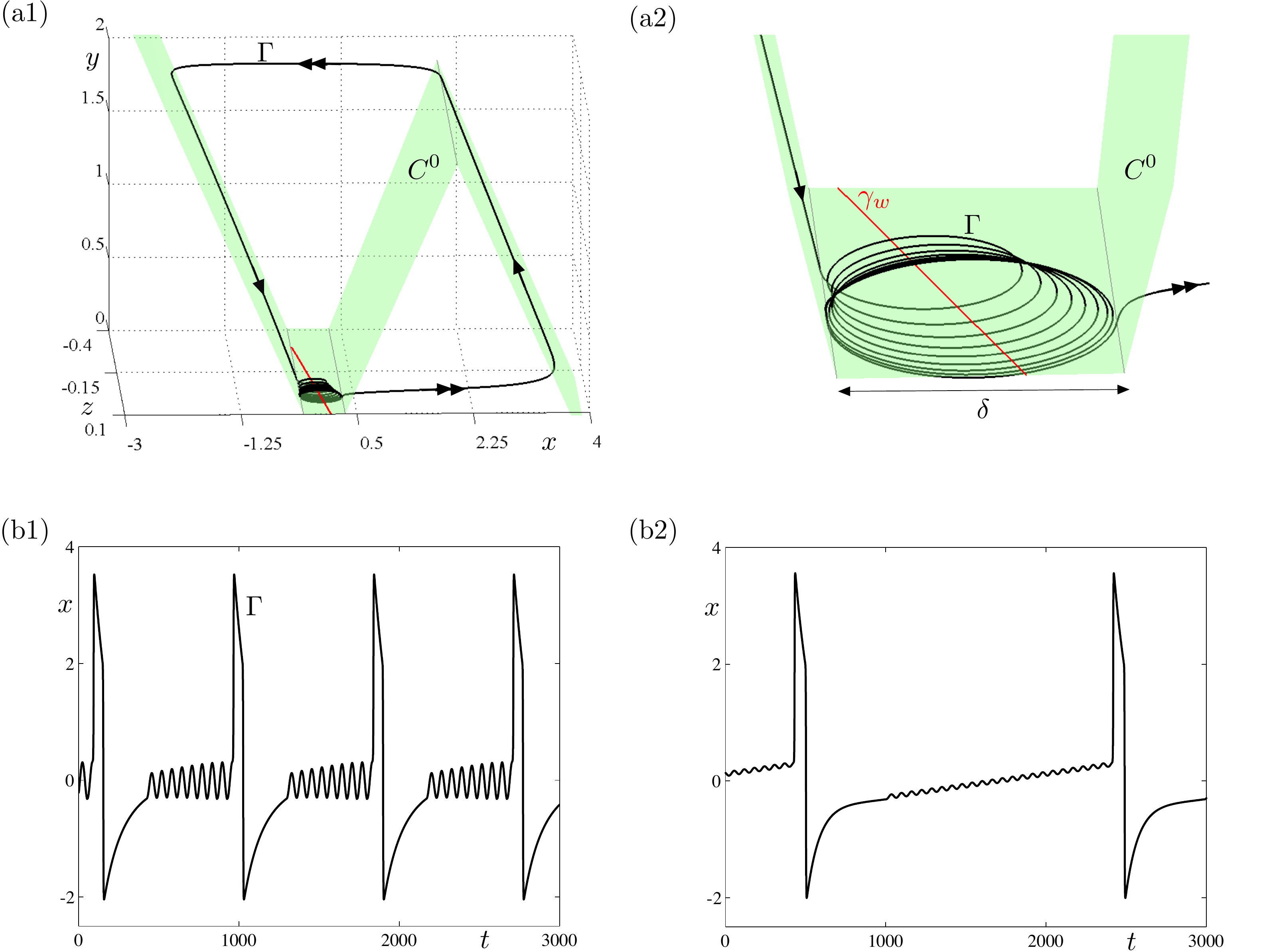}
\caption{\label{mmofig}Periodic PWL MMO $\Gamma$ near a folded node. Panels (a1) and (a2) show a phase-space representation of $\Gamma$ together with the 4-piece PWL critical manifold $C^0$; panel (a2) is a zoom of panel (a1) near the central flat zone, highlighting the SAOs. Panel (b1) shows the time profile of $\Gamma$ for the fast variable $x$. Panel (b2) shows a similar MMO obtained by imposing conditions~\eqref{cond12}, that is, with SAOs having a constant amplitude.}
\end{figure}

An example of periodic MMO obtained with the extended model~\eqref{pwlgloret-xyz}--\eqref{pwlfunc4min} is shown in Figure~\ref{mmofig}. System~\eqref{pwlgloret-xyz}--\eqref{pwlfunc4min} can be seen as a PWL phenomenological neuron model displaying canard-induced MMOs, similar to three-dimensional versions of the FitzHugh-Nagumo model studied, e.g., in~\cite{W05}.
Note that the MMO in this model have SAOs with increasing amplitude as the trajectory travels through the central zone. This is simply due to the fact that the eigenvalues in the central zone have non-zero real part because of the new terms in the $z$-equation and the fact that we have chosen the values of the $\alpha_i$'s without particular constraints. However, by imposing constraints on the $\alpha_i$'s and the $p_i$'s, we can obtain MMO with small oscillations of constant amplitude, namely 
\begin{eqnarray}
\label{cond12}
  \begin{array}{r}
    \alpha_1=\alpha_3 = 0,\\
    p_2\alpha_2 < 0.
  \end{array}
\end{eqnarray}

It is well-known that, in the smooth case, the small oscillations of MMOs near a folded node do not typically have constant amplitude. Rather, the amplitude first decreases and then increases past the folded node. This behaviour can also be captured in a PWL model by splitting the central part of the critical manifold into two zones and putting a dynamics with complex eigenvalues with negative real part in the first zone and positive in the second (the return does not affect this issue). It is important to notice that the qualitative behaviour and an important part of the quantitative one near PWL folded singularities is obtained with the minimal system that we have considered. Adding more zones in the central part would only be useful in order to refine the shape of the SAOs and possibly to fit data. In other words, we can update the minimal model in order to capture more information from the smooth case, such as the shape of the SAOs near a folded node. 

A natural extension of the present work will be to study the model with return \eqref{pwlgloret-xyz}-\eqref{pwlfunc4min} analytically. In particular, by looking at the structure of periodic MMOs in this model using return maps. Among other things, we hope to prove that the possible MMO patterns in a PWL folded node follow a similar Farey arithmetic as in the smooth case.

\section{Conclusion}
\label{conclusion}
In this paper, we have presented new advances in the study of three-dimensional PWL systems with multiple timescales displaying canard solutions. This is recent research topic and we manage to obtain a complete comparison with the smooth case at the level of the local dynamics near folded singularities. The necessity of using a central zone was introduced in Arima {\it et al.}~\cite{ARIMA97} as a way to obtain the correct eigenvalue transition near the fold. We now interpret it in a singular perturbation fashion as a blow-up. This blow-up, whose optimal size can be determined in the PWL framework, proves to be the key to understand the connection from the attracting slow manifold to the repelling one, that is, the possibility for canards to exist. Note that with other unfoldings of the singular limit $\eps=0$ one cannot reproduce the local dynamics near folded singularities; see for instance~\cite{PT13,PTV14} where canard solutions can be found but the link with folded singularities is lost.
Another important result presented in this paper is the necessity for a central zone maintained open in the singular limit. Indeed, the blow-up of the corner vanishes when $\eps=0$ but we need to maintain the central zone artificially open in order to understand the possibilities of passage from one side of the critical manifold to the other, that is, reveal the PWL folded singularities. In this way, we also define the singular weak and strong canards as directions, which echoes their main role in the smooth case, namely being the eigendirections of the folded singularity. Finally, we construct a simple model displaying periodic canard-induced MMOs near a PWL folded node, by using the analysis done on the local dynamics and adding a linear global return. This opens the way to studying MMOs in the context of PWL slow-fast systems.

Using the PWL framework, we reproduce all the dynamics from the smooth case, both qualitatively and quantitatively. In particular, we obtain compatible singular phase portraits and control the maximum winding number, that is, the number of secondary canards. What is more, we also suggest elements that naturally appear in the PWL setting and which could allow to revisit the smooth case. In particular, we mention the role of the axis of rotation, that is, the weak canard, which is not a maximal canard in the PWL framework, and the possibility for SAOs near a folded saddle.

From this work, two main perspectives about using a PWL setting for slow-fast systems can be emphasized. First, a theoretical perspective on revisiting singular perturbation theory with canard dynamics using PWL flows and obtaining a simplified version of it. Indeed, as we show in the context of three-dimensional canard problems, all the important dynamics is preserved, and several key objects exist in a clearer manner. In particular, there is a natural way to defined uniquely a slow manifold and maximal canards. Furthermore, the PWL framework allows by construction to separate the SAOs that are purely linear, that is, those that stay in the central zone, from the last SAO which has a nonlinear behaviour as it explodes when passing into the right zone and following a (maximal) secondary canard. There is a great degree of similarity with the smooth case where most of the SAOs can be understood through a linearisation of the dynamics along the axis of rotation (weak canard) and a reduction to the Weber 
equation, whereas the last SAO undergoes a canard explosion and therefore behaves in a very nonlinear way. So, to paraphrase M. Diener in~\cite{diener84}, the natural {\it biotope} of canards is that of PWL vector field, at least it is the simplest environment in which one can understand them, in which the essence of canard dynamics is preserved and anything else is dropped.

The second perspective that this work suggests is in the direction of neuronal modelling. PWL models of spiking neurons have long been developed and successfully analysed~\cite{mckean70,DK04}. In parallel, neuron models featuring more elaborate behaviour, in particular alternance of sub-threshold oscillations and spikes, have also been developed and, more recently, slow-fast ODEs have been increasingly used to construct such models, the role of folded nodes and canards being pointed out and thought as central. PWL neuron models with canards have gained recent interest~\cite{SK11,RCG12,DFHPT13,FG14} but the link to more elaborate models with canards and folded singularities were missing. With this work, we bridge this small yet important gap by  constructing the correct local dynamics that allows to design a folded node in a three-dimensional PWL system. 

Another important class of neuronal behaviour that is well captured by smooth slow-fast models is bursting. A number of smooth models have shown the importance of canards in bursters, in particular related to spike-adding phenomena. It is then a natural next step to also try to reproduce and revisit these models using PWL systems; there has been already some work in that direction~\cite{DK04} but nothing related to canards. The current work gives a good basis to tackle this problem.

Our short-term objective is to take the PWL approach for neuron models at the level of smooth ODEs which, hopefully, will bring simplified yet accurate models, numerically more tractable and also more amenable to analysis and control. In other words, in the near future we aim to build up biophysical PWL models of neurons, in particular to obtain complete PWL versions of conductance-based neuron models, e.g. the Hodgkin-Huxley (HH) model. Previous works exist in this direction~\cite{D97,DR03}, 
mostly relying on planar two-piece local systems with a global return mechanism. However, a three-dimensional reduction of the HH model (amongst many other examples) is known to produce MMO dynamics~\cite{RW08} due to the presence of an $S$-shaped critical manifold and a folded node. Therefore, we plan to construct a three-piece local system to approximate the voltage nullcline of this 3D reduction of the HH model. Moreover, the differential equations for the gating variables can be made PWL by replacing sigmoidal activation functions by PWL ones (an example is given in~\cite{DR03}). Thus, we can in this way obtain a PWL conductance-based PWL MMO model, which will open the way towards biophysically accurate and mathematically easier neuron models.
\section*{Appendix A}

In this appendix we give an sketch of the proof of the Proposition \ref{propcan}. This result deals with maximal canards, that are orbits connecting the slow manifolds  $S_{\varepsilon}^A$ and $S_{\varepsilon}^R$, in particular the ones that, after starting at $L_{\varepsilon}^A=S_{\varepsilon}^A\cap\{x=-\delta\}$, flow through the central zone to $L_{\varepsilon}^R=S_{\varepsilon}^R\cap\{x=\delta\}$ (see formulas~\eqref{attrepmandelta}). Therefore, we first present geometrical arguments revealing the parameter regions for which we can expect the existence of such orbits. 

From \eqref{attrepmandelta}, the intersection points of the straight lines $L_{\varepsilon}^A$ and $L_{\varepsilon}^R$ with the y-axis and the z-axis, that is $(0,y^*),\;(z^*_A,0)$ and $(0,y^*),\;(z^*_R,0)$, respectively satisfy that
\begin{eqnarray}\label{def:z*}
     y^*&=&-\delta p_1 \varepsilon -\left( p_2 p_3+\delta {p_1}^{2}\right) {\varepsilon}^{2} + O(\varepsilon^3),\nonumber \\
     z_A^*&=& \frac{\delta p_1}{p_2}+p_3 \varepsilon+p_1 p_3 {\varepsilon}^{2} + O(\varepsilon^3),\\
    z_R^*&=& -\frac{\delta p_1}{p_2}-p_3 \varepsilon-p_1 p_3 {\varepsilon}^{2} + O(\varepsilon^3).\nonumber
\end{eqnarray}
If we set $\delta= \pi \sqrt{\varepsilon}$, then the sign of $y^*$, for $\varepsilon$ small enough, depends only on the sign of $p_1$ whereas the sign of $z_A^*$ and $z_R^*$ depends on the sign of $p_1$ and $p_2$. Although the line $L_{\varepsilon}^A$ is contained in the switching plane $\{x=-\delta\}$ and $L_{\varepsilon}^R$ in $\{x=\delta\}$, in Figure \ref{fig:slow_str_lines} we represent the different positions of $L_{\varepsilon}^A$ and $L_{\varepsilon}^R$ in a common $(z,y)$-plane. The rows of Figure \ref{fig:slow_str_lines} correspond to the cases $p_2>0$ and $p_2<0$, the columns to the cases $p_1>0$ and $p_1<0$.

\begin{figure}[!t]
\centering
\includegraphics[width=9cm]{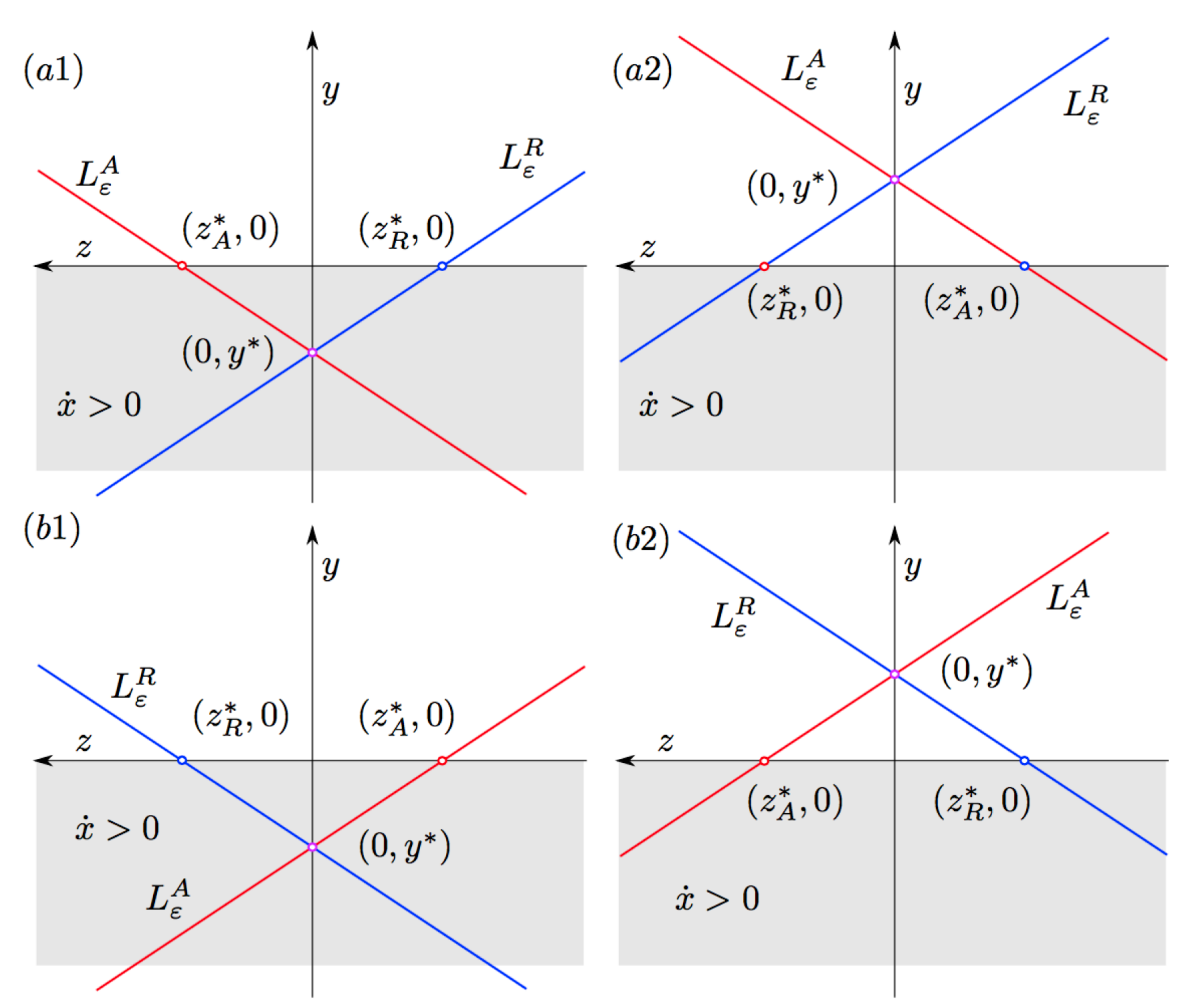}
\caption{Relative position of the straight lines $L_{\varepsilon}^A$ (in red) and $L_{\varepsilon}^R$ (in blue)  represented over a common zy-plane. The shaded half-plane indicates the boundary conditions for which orbits cross the central zone. Different panels correspond to different values of the sign of $p_1$ and $p_2$. The panels in the top row corresponds to the folded-saddle case ($p_2>0$): (a1) when $p_1>0$ and (a2) when $p_1<0$. In both cases one observes the possibility for the existence of orbits crossing through the central zone from $L_{\varepsilon}^A$ to $L_{\varepsilon}^R$. The panels in the bottom row correspond to the folded-node case ($p_2<0$). In panel (b1), when $p_1>0$, we observe that the existence of canards orbits is possible, however, when $p_1<0$, in panel (b2) we observe that the existence of canards orbits is impossible.}\label{fig:slow_str_lines}
\end{figure}

When $p_2>0$ and $p_1>0$ we obtain that $y^*<0$, $z_A^*>0$, $z_R^*<0$, and the straight lines $L_{\varepsilon}^A$ and $L_{\varepsilon}^R$ are located as the ones depicted in Figure \ref{fig:slow_str_lines}(a1). Orbits with initial conditions on the segment of $L_{\varepsilon}^A$ contained in the half-plane $\{x=-\delta, y<0\}$ (shadowed regions in Figure \ref{fig:slow_str_lines}(a1)) flow through the central zone by increasing the coordinate $z$ (since $p_3 > 0$) until they reach the half-plane $\{x=\delta, y<0\}$. We conclude that, in this case, orbits connecting $L_{\varepsilon}^A$ and $L_{\varepsilon}^R$ may exist, that is, in this case we can expect the existence of maximal canards (primary and/or secondary). In a similar way we conclude the possible existence of maximal canards orbits in any of the other cases, except when $p_1<0$ and $p_2<0$, see Figure \ref{fig:slow_str_lines}(b2). In this case, orbits connecting the segment of $L_{\varepsilon}^A$ and 
$L_{\varepsilon}^R$ in the shadowed region by increasing the $z$-coordinate cannot exist.

Once we have explored the parameter regions for which we can expect maximal canards, we deal with the explicit computation of such orbits. That is, we focus on solutions $\varphi(t;\mathbf{p})$ of system~\eqref{3d}-\eqref{pwlfunc} flowing through points $\mathbf{p}\in L_{\varepsilon}^A$ and satisfying that:
\begin{itemize}
 \item [i)] there exists a flight time $t(\mathbf{p})>0$ such that $\varphi(t(\mathbf{p});\mathbf{p})\in L_{\varepsilon}^R$; and,
 \item [ii)] for every $t\in (0,t(\mathbf{p}))$ we have that $|\varphi(t;\mathbf{p})\mathbf{e}_1^T | < \delta$ (where $\mathbf{e}_1^T$ is the transpose of the first vector in the canonical basis of $\mathbb{R}^3$).
\end{itemize} 

Since system~\eqref{3d}-\eqref{pwlfunc} is reversible with respect to the involution $$(x,y,z,t)\to (R(x,y,z),-t)$$ where $R(x,y,z)=(-x,y,-z)$, it follows that $R \varphi(t;\mathbf{p})=\varphi(-t;R \mathbf{p})$. Therefore, if $\gamma_{\mathbf{p}}$ is a maximal canard with flight time $t(\mathbf{p})$ between $L_{\varepsilon}^A$ and $L_{\varepsilon}^R$ and $\mathbf{q}=\varphi(t(\mathbf{p});\mathbf{p})$, then $\gamma_{R \mathbf{q}}$ is also a maximal canard with the same flight time. Under the assumption of uniqueness of maximal canards having equal flight time, it follows that $R \mathbf{q} = \mathbf{p}$, which proves the statement (a) of Proposition \ref{propcan}. 

As we have just proved, a maximal canard with initial condition $(-\delta,y,z) \in L_{\varepsilon}^A$ must also flow through the point $(\delta,y,-z) \in L_{\varepsilon}^R$. We conclude the following extra restrictions for the initial condition $(-\delta,y,z)$, see Figure~\ref{fig:slow_str_lines}:
\begin{equation}\label{eq:restriction}
\begin{tabular}{|l|l|}\hline
Sign of parameters  & restriction on $z$ \\ \hline
 $p_2>0$ and $p_1>0$ & $0>z$ \\ \hline
 $p_2>0$ and $p_1<0$ & $0>z_A^*>z$ \\ \hline
 $p_2<0$ and $p_1>0$ & $0>z \geq z^*_A$\\ \hline
\end{tabular}
\end{equation}
We note that in the three cases the $z$-coordinate is negative. 

Consider $p_1>0$. The solution $\varphi(t,\mathbf{p})$ of system~\eqref{3d}-\eqref{pwlfunc} through a point $\mathbf{p}=(-\delta,y,z)$ has coordinates equal to
\begin{align}\label{eq:flow_p1_p}
x(t)&=\dfrac{p_2 z- p_1 \delta}{p_1} \cos \left( \sqrt{\varepsilon p_1} t \right) 
    +\dfrac{\varepsilon p_2 p_3 - p_1 y}{p_1 \sqrt{\varepsilon p_1}} \sin \left( \sqrt{\varepsilon p_1} t \right)
    -\dfrac{p_2 z+ \varepsilon p_2 p_3 t }{p_1},\nonumber \\
y(t)&=\dfrac{\sqrt{\varepsilon p_1} (p_2z-p_1 \delta) }{p_1}  \sin \left( \sqrt{\varepsilon p_1} t \right)
    -\dfrac{\varepsilon p_2 p_3 - p_1 y_0}{p_1}  \cos \left( \sqrt{\varepsilon p_1} t \right) 
    +\dfrac{\varepsilon p_2 p_3}{p_1},\\
z(t)&=\varepsilon p_3 t +z.\nonumber    
\end{align}
Taking into account $\varphi(t(\mathbf{p});\mathbf{p}) =(\delta, y, -z)$, and isolating $t(\mathbf{p})$ from the third equation in \eqref{eq:flow_p1_p}, it  follows that
\begin{equation}\label{def:flyingtime}
t(\mathbf{p})=-\dfrac {2 z}{\varepsilon p_3}
\end{equation}
and hence the other two equations can be rewritten as
\begin{align}\label{eq:max_canards}
0 &=\dfrac{p_2 z- p_1 \delta}{p_1} \left( \cos \left( - \dfrac {2 \sqrt{p_1} }{p_3 \sqrt{\varepsilon}}z \right) +1  \right)
    +\dfrac{\varepsilon p_2 p_3 - p_1 y}{p_1 \sqrt{\varepsilon p_1}} \sin \left( - \dfrac {2 \sqrt{p_1} }{p_3 \sqrt{\varepsilon}}z \right),\nonumber\\
    &\\
0 &=\dfrac{\sqrt{\varepsilon p_1} (p_2z-p_1 \delta) }{p_1}  \sin \left( - \dfrac {2 \sqrt{p_1} }{p_3 \sqrt{\varepsilon}}z \right)
    -\dfrac{\varepsilon p_2 p_3 - p_1 y}{p_1} \left( \cos \left( - \dfrac {2 \sqrt{p_1} }{p_3 \sqrt{\varepsilon}}z \right) -1 \right).\nonumber
\end{align}
As we are interested in solutions $(z,y)$ of the previous system belonging to the straight line $L_{\varepsilon}^A$, from \eqref{attrepmandelta} we express $y$ as a function of $z$ and $\varepsilon p_2p_3 - p_1 y= \varepsilon \lambda_A^{-1} \left( p_1p_2 z - \delta p_1^2 - p_2p_3 \right)$, where $\lambda_A$ is the fast eigenvalue of the system in the half-space $x<-\delta$, see \eqref{attrslowm}. Since $H(-\delta,y,z)$ (from formula~\eqref{axis}) only vanishes at the weak canard (which does not belong to $L^A_{\varepsilon}$), we can conclude that
$$
  {\varepsilon}^2{\lambda_A}^{-2} \left( p_1p_2 z - \delta p_1^2 - p_2p_3 \right)^2+\varepsilon p_1(p_2 z - \delta p_1)^2 \neq 0
.$$Consequently, the trigonometric functions, in the system above, can be isolated as follows
\begin{align}\label{eq:solution_p1_p}
\cos \left( - \dfrac {2 \sqrt{p_1} }{p_3 \sqrt{\varepsilon}}z \right) &=
	  \dfrac{ \varepsilon (p_1 p_2 z-\delta p_1^2-p_2p_3)^2-\lambda_A^2 p_1(p_2 z - \delta p_1)^2 }
		{ \varepsilon (p_1 p_2 z-\delta p_1^2-p_2p_3)^2+\lambda_A^2 p_1(p_2 z - \delta p_1)^2 },\nonumber \\ & \\
\sin \left( - \dfrac {2 \sqrt{p_1}}{p_3 \sqrt{\varepsilon}}z \right) &=
	- \dfrac{2 \lambda_A \sqrt{\varepsilon p_1}(p_1p_2 z -\delta p_1^2-p_2p_3)(p_2 z - \delta p_1)}
		{ \varepsilon (p_1 p_2 z-\delta p_1^2-p_2p_3)^2+\lambda_A^2 p_1(p_2 z - \delta p_1)^2 }.\nonumber
\end{align}
Solutions of system \eqref{eq:solution_p1_p} are also solutions of equation 
\begin{equation}\label{eq:tag_sol_p1_p}
\tan \left( - \dfrac {2 \sqrt{p_1} }{p_3 \sqrt{\varepsilon}}z \right) =
		 2 |\lambda_A | \sqrt{\varepsilon p_1} \dfrac{(p_1p_2 z -\delta p_1^2-p_2p_3)(p_2 z - \delta p_1)}
		        { \varepsilon (p_1 p_2 z-\delta p_1^2-p_2p_3)^2-\lambda_A^2 p_1(p_2 z - \delta p_1)^2 }.
\end{equation}
On the contrary, there are solutions of equation \eqref{eq:tag_sol_p1_p} which are not solutions of system \eqref{eq:solution_p1_p}. By noting that only half of the solutions of equation \eqref{eq:tag_sol_p1_p} become solutions of system~\eqref{eq:solution_p1_p}, we focus on the number of solutions of equation \eqref{eq:tag_sol_p1_p}, which gives access to the number of canard solutions. 

Let us consider the rational function 
\[
 Q(z)= \dfrac{(p_1p_2 z -\delta p_1^2-p_2p_3)(p_2 z - \delta p_1)}
		        { \varepsilon (p_1 p_2 z-\delta p_1^2-p_2p_3)^2-\lambda_A^2 p_1(p_2 z - \delta p_1)^2 },
\]
which is the quotient of two polynomials of degree two, with opposite concavities both  depending on the sign of $p_1$.
The zeros of $Q(z)$ are 
\begin{eqnarray}\label{del:zeros}
 z_1=\delta \dfrac {p_1}{p_2}, \quad z_2=\delta \dfrac {p_1}{p_2} + \dfrac {p_3}{p_1}
\end{eqnarray}
the poles are 
\begin{eqnarray}\label{del:poles}
 \tilde{z}_1=\delta \dfrac {p_1}{p_2}+\dfrac {p_3}{\sqrt{p_1}}\sqrt{\varepsilon}+O(\varepsilon), \quad \tilde{z}_2=\delta \dfrac {p_1}{p_2}-\dfrac {p_3}{\sqrt{p_1}}\sqrt{\varepsilon}+O(\varepsilon),
\end{eqnarray}
and the asymptotic behaviour of $Q(z)$ when $z$ tends to $\pm\infty$ is $(\varepsilon p_1 - \lambda_A^2)^{-1} \approx -1$. The qualitative behaviour of the graph of $Q(z)$ together with the numerator and the denominator of $Q(z)$ is the depicted in Figure \ref{fig:graphQ}, depending on the sign of $p_2$.

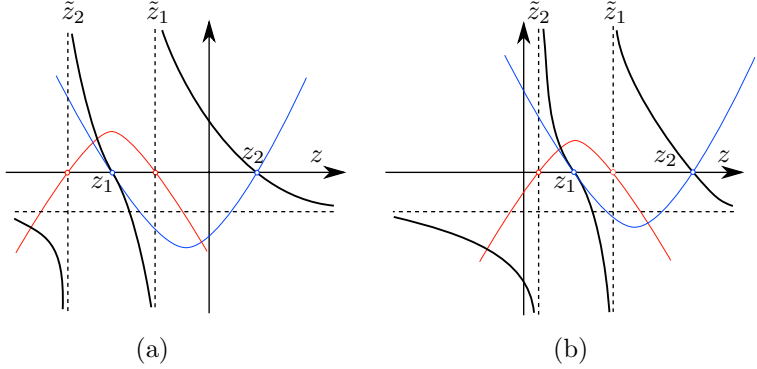
\begin{figure}[!t]
 \begin{center}
 \includegraphics[width=10cm]{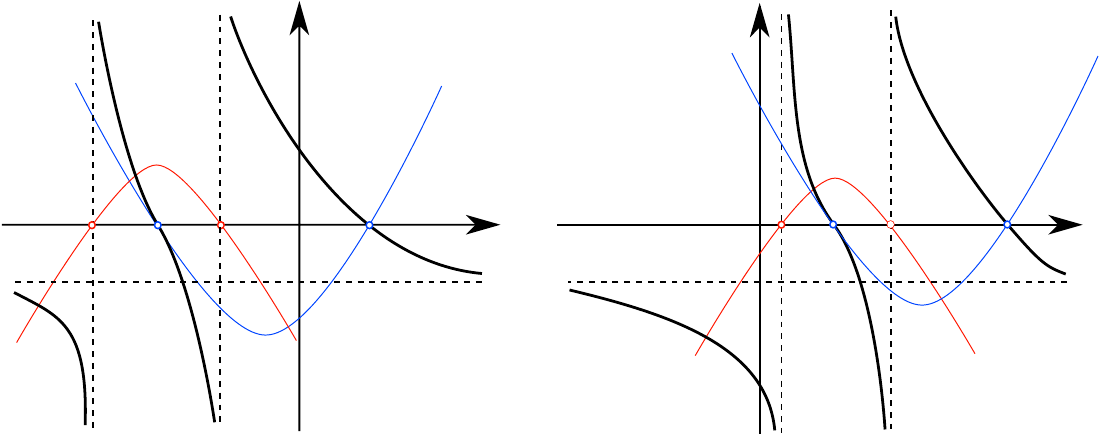}
 \put(-235,-15){(a)}
 \put(-169,59){$z$}
 \put(-263,113){$\tilde{z}_2$}
 \put(-252,49){$z_1$}
 \put(-195,60){$z_2$}
 \put(-231,113){$\tilde{z}_1$}
 \put(-77,-15){(b)}
 \put(-15,59){$z$}
 \put(-77,49){$z_1$}
 \put(-39,59){$z_2$}
 \put(-87,113){$\tilde{z}_2$}
 \put(-58,113){$\tilde{z}_1$}
 \caption{Graph of the numerator (blue), the denominator (red) and the rational function $Q(z)$ (black): (a) for $p_2<0$; (b) for $p_2>0$.}\label{fig:graphQ}
 \end{center}
\end{figure}

Furthermore, the zeros and the vertical asymptotes of the $\pi$-periodic function $\tan \left( - \frac {2 \sqrt{p_1} }{p_3 \sqrt{\varepsilon}}z \right)$ are 
\begin{equation}\label{rkrktilde}
r_k=-k \dfrac {p_3 \sqrt{\varepsilon}}{\sqrt{p_1}} \dfrac{\pi}{2}\quad \text{and} \quad  \tilde{r}_k = -\dfrac {2k+1}{2} \dfrac {p_3 \sqrt{\varepsilon}}{\sqrt{p_1}} \dfrac{\pi}{2}, \quad k=0,1,2,\ldots, 
\end{equation}
respectively.

Consider $p_2<0$. For $\varepsilon$ small enough it follows that $z_A^* < \tilde{z}_1$, see \eqref{def:z*} and \eqref{del:poles}. Since $z_A^*$ is a lower bound for the $z$-coordinate of initial condition of maximal canards, see table \eqref{eq:restriction}, we restrict ourselves to the zeros of \eqref{eq:tag_sol_p1_p} in $(\tilde{z}_1,0)$. The function $2 |\lambda_A| \sqrt{\varepsilon p_1} Q(z)$ is positive in $(\tilde{z}_1,0)$, see Figure \ref{fig:graphQ}(b). Therefore, in this interval, the zeros of \eqref{eq:tag_sol_p1_p} are contained each one into each of the following intervals $(\tilde{r}_k,r_k)$ with $k=0,1,\ldots,N,$ where $N$ is such that 
$\tilde{z}_1 \in \left( \tilde{r}_{N},\tilde{r}_{N-1} \right)$. Recall that only half of these zeros correspond to initial conditions of maximal canards, in particular the ones with $k$ odd. Note that the right-hand side of the first equation in  \eqref{eq:solution_p1_p} is negative in $(\tilde{z}_1,0)$. We conclude that there exist 
$[N/2]+1$ zeros of \eqref{eq:tag_sol_p1_p}. By using the fact that $\tilde{z}_1 \in \left( \tilde{r}_{N},\tilde{r}_{N-1} \right)$, together with formulas from~\eqref{del:poles} and~\eqref{rkrktilde}, we can show that, for $\varepsilon$ small enough,
\[
  -\dfrac {4+\pi}{4\pi} < \dfrac {N}{2} - \dfrac {p_1\sqrt{p_1}}{|p_2|p_3} = \dfrac {N}{2} - \mu < \dfrac {\pi-4}{4\pi}.
\]
Then, for every $k=0, 1, \ldots [\mu]$ there exists an orbit $\gamma_k$ through the initial condition $\mathbf{p}_k=(-\delta, y_k, z_k)$ with $z_k \in \left( \tilde{r}_{2k+1},r_{2k+1} \right)$. The approximations appearing in statement (b) of Proposition \ref{propcan} are obtained by using $z_k\approx {r}_{2k+1}$ and $\mathbf{p}_k \in L_{\varepsilon}^A$. In order to prove that $\gamma_k$ is a maximal canard, we need to show that the piece of $\gamma_k$ flowing from $\mathbf{p}_k$ to $R(\mathbf{p}_k)$ is fully contained in the central zone $\{|x|\leq \delta\}$. We conclude this by noting that $\gamma_k$ is contained in a cylinder, whose projection onto the $xz$-plane has amplitude equal to 
\begin{equation}\label{eq.amp.cyl}
\Delta=\delta \left(
		1+\dfrac {p_2 p_3}{p_1 \sqrt{p_1}}\left(k+\dfrac 1 2 \right)
	   \right),
\end{equation}
where $\Delta < \delta$, see Figure \ref{fig.max_true_canard_fn}(a).
Finally, the number of rotations of the maximal canard $\gamma_k$ around the weak canard $\gamma_w$ is 
\[
 2k\pi < -\dfrac{2\sqrt{p_1}}{p_3\sqrt{\varepsilon}}z_k  < 2k\pi + \dfrac {\pi} 2.
\]
In particular $\gamma_0$ is the unique maximal canard which runs less than a turn around $\gamma_w$. Then, $\gamma_0$ is the strong canard.  This concludes the proof of the statement (b) of Proposition \ref{propcan}.

If $p_2>0$, the graph of $2 |\lambda_A| \sqrt{\varepsilon p_1} Q(z)$ is qualitatively represented in Figure \ref{fig:graphQ}(a). Therefore, this function intersects with the graph of $\tan \left( - \frac {2 \sqrt{p_1} }{p_3 \sqrt{\varepsilon}}z \right)$ at infinitely many values $z_k \in \left(r_{k+1},\tilde{r}_k \right)$. In particular the ones with $k$ even. The expression of $z_k$ given in the proposition is obtained by taking $z_k\approx {r}_{2k+1}$. In order to prove that the orbit $\gamma_k$ through the initial condition $\mathbf{p}_k=(-\delta,y_k,z_k)$ corresponds with a maximal canard, we proceed as in the folded node case, that is, by analysing if $\gamma_k$ also flows through the point $R(\mathbf{p}_k)$.

From expression \eqref{eq.amp.cyl}, the amplitude $\Delta$ of the invariant cylinder containing $\gamma_k$ satisfies that $\Delta > \delta$. Since $\gamma_k$, with $k\geq 1$, rotates $k$ times around the faux canard $\gamma_f$, we conclude that $\gamma_k$ leaves the central zone before approaching $R(\mathbf{p}_k)$, see Figure  \ref{fig.max_true_canard_fn}(b). 
Therefore, only $\gamma_0$ becomes a maximal canard, namely the strong canard. 
This proves statement (c) of Proposition \ref{propcan}.
\begin{figure}[!t]
\centering
\includegraphics[width=9cm]{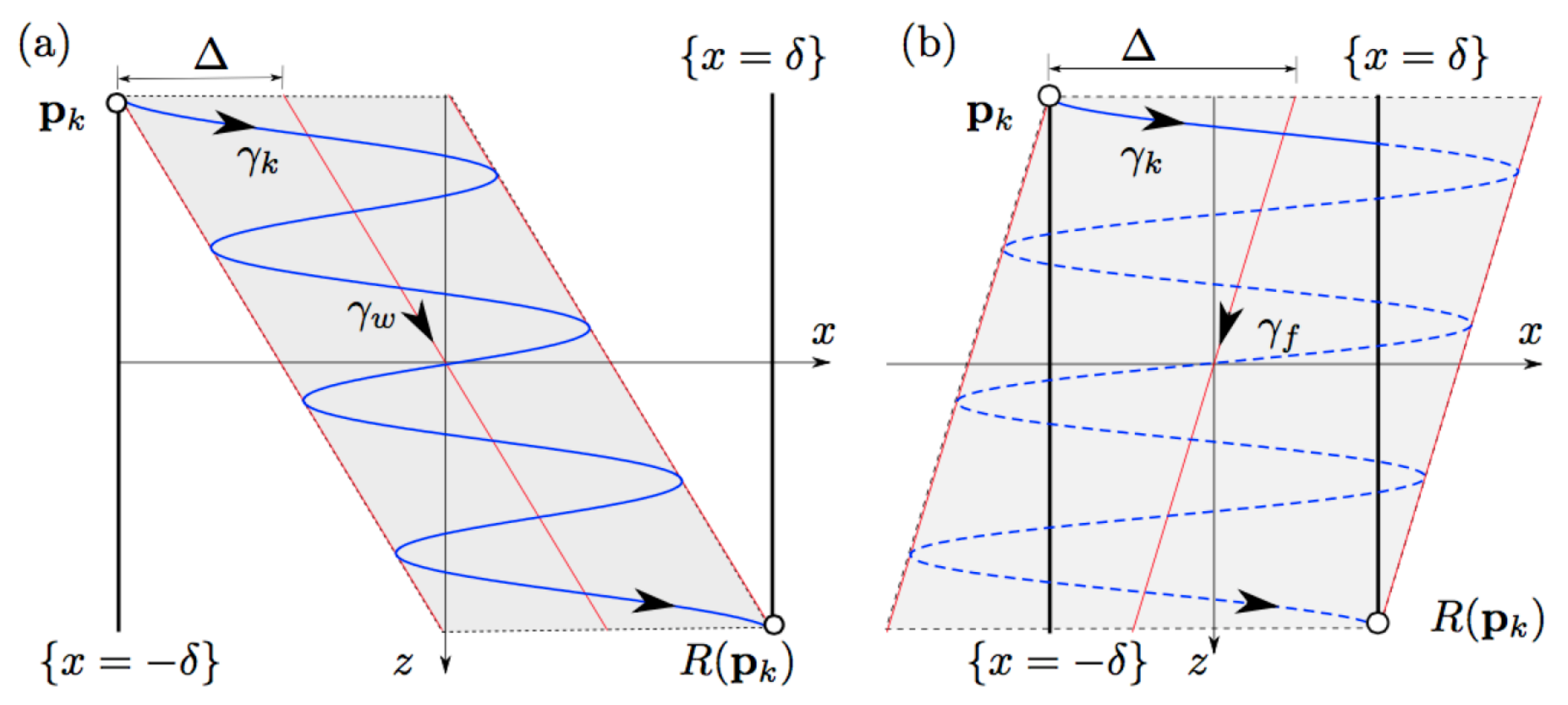}
\caption{\label{fig.max_true_canard_fn}(a) A secondary maximal canard $\gamma_k$ in the folded node configuration ($p_1>0, p_2<0,\ p_3>0$) flowing from the initial condition $\mathbf{p}_k$ in the switching plane $\{x=-\delta\}$ to the reversible point $R(\mathbf{p}_k)$ in the switching plane $\{x=\delta\}$. The grey region is the invariant cylinder given by \eqref{axis} with axis being the weak canard $\gamma_{w}$ and amplitude $\Delta$ given by \eqref{eq.amp.cyl}. (b) Orbit through $\mathbf{p}_k=(-\delta,y_,z_k)$ with $z_k$ a solution of \eqref{eq:tag_sol_p1_p} and $k\geq 1$. In the particular case depicted $k=3$. Notice that $\gamma_k$ leaves the central zone before reaching $R(\mathbf{p}_k)$, the dashed curve is not part of the orbit. The grey region corresponds with the invariant cylinder whose axis is the faux canard $\gamma_f$.}
\end{figure}

Consider now $p_1<0$, and suppose that $p_2>0$ (as we have seen above, if $p_2<0$ there are no maximal canards; see panel (b2) of Figure~\ref{fig:slow_str_lines}). The solution of system~\eqref{3d}-\eqref{pwlfunc} through a point $(x,y,z)$ has coordinates equal to
\begin{align*}
x(t)&= \frac{{e}^{\sqrt{\varepsilon \left| p_1\right| }\,t}\,\left( \sqrt{\varepsilon \left| p_1\right| }\,p_2\,z +\left| p_1\right| \,y-\sqrt{\varepsilon}\, 
	    {\left| p_1\right| }^{\frac{3}{2}}\,x+\varepsilon\,p_2\,p_3\right)}
	    {2\,\sqrt{\varepsilon \left| p_1\right| } |p_1| }\\
     &+ \frac{{e}^{-\sqrt{\varepsilon \left| p_1\right| }\,t}\,\left( -\sqrt{\varepsilon  \left| p_1\right| }\,p_2\,z+\left| p_1\right|\,y+\sqrt{\varepsilon}\,{\left| p_1\right| }^{\frac{3}{2}}\,x+\varepsilon\,p_2\,p_3\right) } {2\,\sqrt{\varepsilon \left| p_1\right| } |p_1| }
     +\frac{p_2\,z+\varepsilon\,p_2\,p_3\,t}{\left| p_1\right| }\\
y(t)&=-\frac{{e}^{\sqrt{\varepsilon \left| p_1\right| }\,t}\,\left( \sqrt{\varepsilon \left| p_1\right| }\,p_2\,z+\left| p_1\right|\,y-\sqrt{\varepsilon}\,{\left| p_1\right| }^{\frac{3}{2}}\,x+\varepsilon\,p_2\,p_3\right) }
	    {2\,\left| p_1\right| }\\
    &+\frac{{e}^{-\sqrt{\varepsilon \left| p_1\right| }\,t}\,\left( -\sqrt{\varepsilon \left| p_1\right| }\,p_2\,z+\left| p_1\right|\,y+\sqrt{\varepsilon}\,{\left| p_1\right| }^{\frac{3}{2}}\,x_0+\varepsilon\,p_2\,p_3\right) }
	    {2\,\left| p_1\right| }
    -\frac{\varepsilon\,p_2\,p_3}{\left| p_1\right| }\\    
z(t)&=\varepsilon p_3 t +z.    
\end{align*}
Proceeding as in the previous case, solutions flowing from $(-\delta,y,z)$ to $(\delta,y,-z)$ satisfy that 
\begin{align}\label{sinhcosh}
\nonumber \sinh \left( -\dfrac{2\sqrt{|p_1|}}{p_3 \sqrt{\varepsilon} } z \right) 
    &= - \dfrac {(\varepsilon p_2 p_3 +|p_1|y)^2-\varepsilon p_1 (p_2 z + |p_1| \delta)^2}
	         {(\varepsilon p_2 p_3 +|p_1|y)^2+\varepsilon p_1 (p_2 z + |p_1| \delta)^2}, \\ & \\
\nonumber \cosh \left( -\dfrac{2\sqrt{|p_1|}}{p_3 \sqrt{\varepsilon} } z \right) 
    &= - \dfrac {2 \sqrt{\varepsilon |p_1|}(\varepsilon p_2 p_3 +|p_1|y)(p_2 z + |p_1| \delta) }
	         {(\varepsilon p_2 p_3 +|p_1|y)^2+\varepsilon p_1 (p_2 z + |p_1| \delta)^2}.
\end{align}
Taking into account that $|p_1|= -p_1$ and  $(-\delta,y,z) \in L_{\varepsilon}^A$, that is, $\varepsilon p_2p_3 - p_1 y= \varepsilon \lambda_A^{-1} \left( p_1p_2 z - \delta p_1^2 - p_2p_3 \right)$, solutions of system~\eqref{sinhcosh} are also solutions of the equation
\begin{eqnarray}\label{eq:tagh_p}
 \tanh \left( -\dfrac{2\sqrt{|p_1|}}{p_3 \sqrt{\varepsilon} } z \right) = \dfrac {1}{2\sqrt{\varepsilon |p_1|} 
 |\lambda_A |} \dfrac {1}{Q(z)}.
\end{eqnarray}
Notice that the numerator of the rational function $Q(z)^{-1}$ is always positive and therefore has no zeros. The poles of $Q(z)^{-1}$ coincide with the zeros of $Q(z)$ given in \eqref{del:zeros}. Direct derivation of the function $Q(z)^{-1}$ gives a rational function where the numerator is a quadratic polynomial and the denominator does not vanish. Then, we obtain that this function has its extreme values at
\[
 m_1=\delta\dfrac {p_1}{p_2} + \dfrac {p_3}{\sqrt{|p_1|}} \sqrt{\varepsilon}, \quad  m_2=\delta\dfrac {p_1}{p_2} - \dfrac {p_3}{\sqrt{|p_1|}} \sqrt{\varepsilon},
\]
the former being a local maximum and the latter being a local minimum. Moreover, $Q(m_2)^{-1}=2\sqrt{\varepsilon|p_1|} |\lambda_A|$ and $Q(m_1)^{-1}=-2\sqrt{\varepsilon|p_1|} |\lambda_A|$. The qualitative behaviour of the graph of $Q(z)^{-1}$ is represented in Figure \ref{fig:maximal_canards2}. Since $0< \tanh\left( -\frac{2\sqrt{|p_1|}}{p_3 \sqrt{\varepsilon} } z \right)<1 $ in $(-\infty,0)$, we conclude that this function does not intersect with $Q(z)^{-1}$ in this interval, which proves the statement (d) of Proposition \ref{propcan}.

\begin{figure}[!t]
 \begin{center}
 \includegraphics{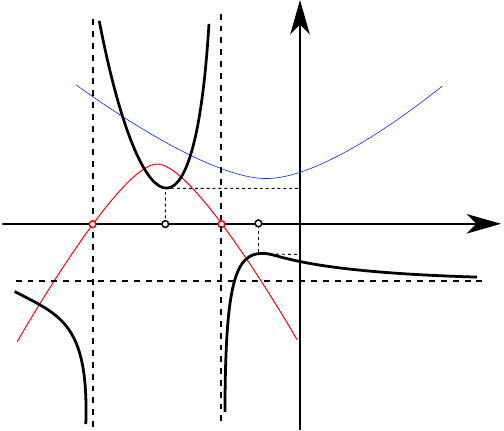}
 \put(-130,64){$z_2$}
 \put(-110,52){$m_2$}
 \put(-92,52){$z_1$}
 \put(-78,63){$m_1$}
 \put(-15,64){$z$}
 \caption{Graph of the numerator (blue), the denominator (red) and the rational function $Q(z)^{-1}$ (black) if $p_1<0$ and $p_2>0$. The values of the function at the extrema $m_2$ and $m_1$ are $\pm 2\sqrt{|p_1| \varepsilon} |\lambda_{\varepsilon}^A|$, respectively. }\label{fig:maximal_canards2}
 \end{center}
\end{figure}
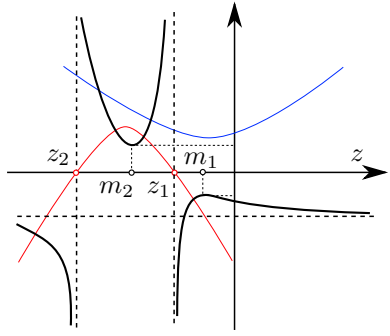

\section*{Appendix B}
In this appendix we show a sketch of the proof of Proposition \ref{selected}.

As we have proved in the previous appendix, maximal canards through points $\mathbf{p}=(-\delta,y,z)$ contained in the switching manifold $\{x=-\delta\}$, are characterized by the solutions of equation \eqref{eq:max_canards} together with $\mathbf{p}\in L^A_{\varepsilon}$. A trivial solution of \eqref{eq:max_canards} is obtained by setting 
\begin{equation}\label{eq:max_canards_selec}
    \cos\left(\frac{-2z\sqrt{p_1}}{\sqrt{\eps}p_3}\right) = -1,\quad
    \sin\left(\frac{-2z\sqrt{p_1}}{\sqrt{\eps}p_3}\right) =  0,
\end{equation}
which implies 
\[
  y=y_k = \varepsilon \dfrac {p_2p_3}{p_1},\quad z=z_k=-(2k+1)\dfrac {\pi}{2} \dfrac{p_3}{p_1}\sqrt{\varepsilon p_1} \quad k\in \mathbb{Z}.
\]
Using the conditions~\eqref{eq:max_canards_selec} and the fact that the $z$-coordinate in the switching plane $\{x=-\delta\}$ has to be negative, we conclude that $\delta = \delta_k$ for $k\in \mathbb{N}$, where
\begin{eqnarray*}
\delta_k= -\frac{p_2p_3}{p_1^2}\left((2k+1)\frac{\pi}{2}\sqrt{\eps p_1}+1 \right). 
\end{eqnarray*}
Substituting the values of $y=y_k$, $z=z_k$ and $\delta=\delta_k$ in \eqref{eq:flow_p1_p} we obtain the explicit time parametrization of the strong canard for $k=0$ and of the $k$th secondary canard for $k\geq 1$%
\begin{eqnarray*}
x_{\mathrm{sc}}(t) &=& \frac{p_2p_3}{p_1^2}\cos(\sqrt{\eps p_1}t)-\sqrt{\eps p_1}\frac{p_2p_3}{p_1^2}\left(\sqrt{\eps p_1}t 
  - \left(k+\frac{1}{2}\right)\pi\right),\\
y_{\mathrm{sc}}(t) &=& -\sqrt{\eps p_1}\frac{p_2p_3}{p_1^2}\big(\sin(\sqrt{\eps p_1}t)-\sqrt{\eps p_1}\big)\\
z_{\mathrm{sc}}(t) &=& ~~\sqrt{\eps}p_3\left(\sqrt{\eps}t - \left(k+\frac{1}{2\sqrt{p_1}}\right)\pi\right) 
\end{eqnarray*}
for $0\leq t \leq t(\mathbf{p})=(2k+1)\pi \frac 1{\sqrt{\varepsilon p_1}}$, see \eqref{def:flyingtime}. This ends the proof of Proposition~\ref{selected}.

\bibliography{dgpprt}
\bibliographystyle{siam}

\end{document}